\documentclass{article}

\usepackage{mytemplate}

\newtheorem{thm}{Theorem}
\newtheorem{lem}{Lemma}[section]
\newtheorem{prop}[lem]{Proposition}
\newtheorem{cor}[lem]{Corollary}
\newtheorem{conj}[lem]{Conjecture}

\newcommand{\fqr}{\F_{q^r}}
\newcommand{\fqx}{\F_q[x]}
\newcommand{\fqrp}{\F_{q^{r/p}}}
\newcommand{\fqsp}{\F_{q^{s/p}}}
\newcommand{\ffqs}{\F_{q^s}}
\newcommand{\fqu}{\F_{q^u}}
\newcommand{\fqv}{\F_{q^v}}

\newcommand{\fq}{\F_q}
\newcommand{\fp}{\F_p}
\newcommand{\fd}{\mathcal{F}_d}
\newcommand{\gd}{\mathcal{G}_d}
\newcommand{\fe}{\mathcal{F}_e}
\newcommand{\hd}{\mathcal{H}_d}
\newcommand{\fqs}{{\F_q^\times}}
\renewcommand{\d}{\partial}
\renewcommand{\t}{\mathrm{t}}
\newcommand{\pb}{\bar{\psi}}
\newcommand{\U}{\mathbf{U}}
\newcommand{\tqrp}{\t_{q^r/p}}
\newcommand{\tqrq}{\t_{q^r/q}}
\newcommand{\tqsp}{\t_{q^s/p}}
\newcommand{\tqkp}{\t_{q^m/p}}
\newcommand{\tqp}{\t_{q/p}}
\newcommand{\afd}[1]{{\la #1\ra}_{f\in\fd}}
\newcommand{\hW}{{\hat{V}}}
\newcommand{\hw}{\hW}
\newcommand{\hv}{\hw}

\newcommand{\hwdt}{{\hat{v}_{d,\th}}}

\newcommand{\sr}{\mathcal{S}(\R)}
\newcommand{\srr}{\mathcal{S}(\R^2)}
\newcommand{\M}{\mathcal{M}}
\newcommand{\p}{\mathcal{P}}
\newcommand{\dz}{\frac{\dd}{\dd z}}
\newcommand{\usp}{\mathbf{USp}}
\newcommand{\epr}{e_{p,r}}
\newcommand{\eeps}{e_{p,s}}
\newcommand{\pr}{\mathrm{pr}}
\renewcommand{\O}{\mathcal{O}}
\newcommand{\od}{{\O_d}}
\newcommand{\aod}[1]{{\la {#1}\ra}_{f\in\od}}
\renewcommand{\M}{\mathcal{M}}
\newcommand{\et}[1]{{\eta\lb #1\rb}}
\newcommand{\m}{^{(m)}}
\newcommand{\qrp}{\frac{q^r}{p}}

\begin{document}

\title{On the Distribution of Zeroes of Artin-Schreier L-functions}
\author{Alexei Entin}
\maketitle

\begin{abstract} We study the distribution of the zeroes of the L-functions of curves in the Artin-Schreier family.
We consider the number of zeroes in short intervals and obtain partial results which agree with a random unitary matrix model.
\end{abstract}

\section{Introduction and statement of main results}\label{intro}

Let $p$ be a prime number, $q=p^n$ its power.
Let $d$ be a natural number prime to $p$. We consider the family of curves over $\F_q$ defined by
an equation of the form \beq\label{as}y^p-y=f(x)=a_dx^d+...+a_1x\eeq with $a_i\in\F_q,a_d\neq 0$ and $a_k=0$ for all $k$ divisible by $p$
(every curve defined by an equation of the form $y^p-y=f(x)$ with $f\in\fq[x]$ of degree $d$ is a twist of a curve of the form \rf{as} satisfying
this condition, see section \ref{geom}).
We call such curves Artin-Schreier curves, or A-S curves for short, and the corresponding family the A-S family (with parameter $d$).

Denote by $\Psi$ the set of nontrivial additive characters of $\F_p$. It is known that the L-function of the normalisation of the projective
closure of a curve defined by \rf{as} factors into primitive L-functions as follows: \beq\label{pr}L_f(z)=\prod_{\psi\in\Psi}L_{f,\psi}(z)\eeq
with \beq\label{deflpsi}L_{f,\psi}(z)=\exp\lb\sum_{r=1}^\infty\sum_{\al\in\fqr}\psi\lb\tr_{\F_{q^r}/\F_p}f(\al)\rb\frac{z^r}{r}\rb.\eeq
Each of the $p-1$ factors in \rf{pr} is a polynomial of degree $d-1$ with all zeroes having absolute value $q^{-1/2}$ due to the
Riemann Hypothesis for curves over finite fields (see section \ref{lfunc}).

Denote by $\fd$ the set of polynomials $f$ of the form in \rf{as} satisfying the stated conditions.
We denote $$T_{f,\psi}^r=\sum_{i=1}^{d-1} \rho_i^r,$$ where $\rho_i$ are the normalised zeroes of $L_{f,\psi}$ counting multiplicity.
The zeroes are normalised as follows: $\rho_i=q^{1/2}\lam_i^{-1}$, where $\lam_i$ are the zeroes of $L_{f,\psi}$. We have $|\rho_i|=1$.
Note that the normalised zeroes are proportional to the inverse zeroes of $L_{f,\psi}(z)$ and we preserve this normalisation convention for zeroes
of L-functions throughout the paper.
The quantity $T_{f,\psi}^r$ is the trace of the $r$-th power of the Frobenius element corresponding to the L-function $L_{f,\psi}$. 
For any finite set $A$ and a function $X:A\to\C$ (we denote by $\Z,\Q,\R,\C$ the set of integers, rational, real and complex numbers respectively)
we denote by $\la X(a)\ra_{a\in A}$ the average of $X(a)$ as $a$ runs uniformly through
$A$, in other words $$\la X(a)\ra_{a\in A}=\#A^{-1}\sum_{a\in A}X(a).$$
Denote $$M_d^r=\afd{T_{f,\psi}^r}.$$ It does not depend on the choice of $\psi\in\Psi$, see
section \ref{asfam}. 
Our main result is the following

\begin{thm}\label{thm1}$$M_d^r=-e_{p,r}q^{r/p-r/2}+O\lb rq^{r/2-(1-1/p)d}+q^{-r/2}\rb,$$
where \beq\label{e} e_{p,r}=\choice{0,}{(r,p)=1,}{1,}{p|r}\eeq (the implicit constant is absolute, i.e. does not
depend on $p,q,r,d$).\end{thm}

Note that the error term in Theorem \ref{thm1} is small when $r\le (2-2/p-\eps)d$ for any fixed $\eps>0$.
For $p>2$ Theorem \ref{thm1} suggests that for $d\to\infty$ the zeroes of all the L-functions in this family when taken together 
are distributed quite uniformly on the unit circle.
We conjecture that the average number of zeroes of $L_{f,\psi}$ contained in an arc of length $O(1/d)$ on the unit circle as $f$ varies uniformly through $\fd$ tends to the length of the arc divided by $2\pi$ as $d\to\infty$. For arcs of length $l$ s.t. $ld\to\ity$ this has recently been proved in \cite{uniform}.
The conjecture is related to the random unitary matrix model for A-S L-functions which we present in section \ref{rmm}.
We are only able to obtain a weaker result with the arc replaced by a window function with bounded frequency.
Denote by $\sr$ the space of smooth complex-valued functions on the real line with all derivatives decaying faster than any power of $t$ at
infinity (the Schwartz space).

\begin{thm}\label{thm2} Assume $p>2$. Let $V\in\sr$ be a function s.t. its Fourier transform 
$$\hat{V}(s)=\frac{1}{2\pi}\inti V(t)e^{-ist}dt$$ is supported
on the interval $(-2+2/p,2-2/p)$. 
Denote $$v_d(t)=\sum_{n=-\infty}^\infty V(d(t+2\pi n)).$$ This function has period $2\pi$.
Let $\th$ be any real number.
Denote $$S_f=\sum_{j=1}^{d-1}v_d(\th_j-\th),$$ where $\rho_j=e^{i\th_j}$ are the normalised zeroes of $L_{f,\psi}$,
$\th_j$ being real numbers well defined modulo $2\pi$.
Then $$\afd{S_f}=\frac{1}{2\pi}\int_{-\infty}^\infty V(t)dt+o(1)$$ as $d\to\infty$
and $q$ may vary as we please, i.e. there is a bound on the decay rate of the $o(1)$ term which depends on $d$ and $V$ but not on $q$
(in fact the error becomes smaller as $q$ grows).
\end{thm}

We conjecture that this in fact holds for any $V\in\sr$, but with the presently
existing methods it seems difficult to prove.

We also consider some nonlinear statistics of the zeroes of L-functions in the A-S family. Let $V\in\srr$ be a two variable window function and $v_{d-1}(t,u)$ the periodic window function associated with $V$
by $$v_{d-1}(t,u)=\sum_{m,n=-\ity}^\ity V((d-1)(t+2\pi m),(d-1)(u+2\pi n)).$$ Let $\rho_1,...,\rho_{d-1}$ be the normalised zeroes of $L_{f,\psi}$
for $f\in\fd$, $\rho_j=e^{i\th_j}$ and let $\th$ be some fixed real number.
We consider the
2-level density function (at $\th$):
$$S^2_\th(f,\psi)=\sum_{1\le j,k\le N\atop{j\neq k}}v_{d-1}(\th_j-\th,\th_k-\th).$$
For a function $V\in\srr$ we define its Fourier transform by $$\hv(\eta,\xi)=\frac{1}{4\pi^2}\inti\inti V(t,u)e^{-i\eta t-i\xi u}\dd t\dd u.$$
In section \ref{nonlinear} we will prove the following

\begin{thm}\label{thm2p}Assume $p>2$. Let $V\in\srr$ be a window function s.t. its Fourier transform $\hv(\eta,\xi)$ is supported on the set $|\eta|+|\xi|\le 1$.
Then for all $\th$, $$\afd{S^2_\th(f,\psi)}= \frac{1}{4\pi^2}\inti\inti V(t,u) \lb \frac{\sin((t-u)/2)}{(t-u)/2}\rb^2\dd t\dd u +o(1)$$ as $d\to\infty$.\end{thm}

The connection with the random unitary matrix model is discussed in section \ref{nonlinear}.

We also consider the subfamily of A-S curves defined by \rf{as} with odd polynomial $f(x)$. For an odd natural number $d$ denote
by $\od$ the subset of $f\in\fd$ s.t. $f(x)=-f(-x)$. In section \ref{odd} we prove the following
\begin{thm}\label{oddthm} Assume $p>2$ and $d$ is odd. There exists a positive constant $C$ (in fact any $C<1$ will do) s.t. 
$$\aod{T_{f,\psi}^r}=-e_{2,r}+O\lb rq^{-r/6}\rb,$$ provided that $r<Cp\log_q d$ and $r<d/4$
($e_{2,r}$ is defined by \rf{e}).\end{thm}
This result agrees with a random symplectic matrix model for the L-zeroes in the family $\od$ (see section \ref{odd}). 

We consider a more general type of families of L-functions corresponding to Dirichlet characters and show that the A-S family (as well as the A-S family with odd
polynomials) is a special case. We also indicate how our results for the A-S family can be generalised to such families.
This occupies section \ref{secdirg}.

We also consider a related problem - the distribution of the number of points of a curve varying uniformly in a family of A-S curves. 
The proofs of our results are
presented in section \ref{numpoints}. This part is independent of the rest of our work and the interested reader may skip to section \ref{numpoints}
after section \ref{back}. We consider the distribution of the number of points on the curve $C_f$ as $f$ varies uniformly through the A-S family and
$d\to\infty$. Here we denote by $C_f$ the normalisation of the projective closure of the curve defined by \rf{as} for $f\in\F_q[x]$. We define $\gd$ to be the set of all monic degree $d$ polynomials in $\fq[x]$.
For the problem of the distribution of the number of points it is more convenient to consider the family 
of A-S curves defined by the polynomials in $\gd$.
It is not difficult to adjust the statements and proofs for the case of the family $\fd$. 

For the rest of this section let $r$ be a fixed natural number.
We will see in section \ref{geom} that the number of $\F_{q^r}$-rational 
points on an A-S curve $C_f$ always equals 1 modulo $p$ and so we denote $N(f)=(\#C_f\lb\fqr\rb-1)/p$. Our results concern the distribution of $N(f)$
as $f$ varies uniformly in the family $\gd$.

We denote by $\pi(e)$ the number of monic degree $e$ irreducible
polynomials in $\F_q[x]$. We denote by $B(t)$ the Bernoulli random variable which assumes 1 with probability $t$ and 0 with probability $1-t$.
For two random variables $X,Y$ we write $X\sim Y$ if they have the same distribution.

\begin{thm}\label{t1} Let $p,n$ be fixed. For each $e|r$ let $X_{e,1},...,X_{e,\pi(e)}\sim B(1/p)$ be random variables with all $\left\{X_{e,i}\right\}_{e|r,1\le i\le\pi(e)}$ independent.
Then for $d\ge q^r$ the following holds:\bi 
\item[(i)] If $(r,p)=1$, the distribution of $N(f)$ equals that of $$\sum_{e|r}e\sum_{i=1}^{\pi(e)}X_{e,i}.$$
In particular the mean value of $N(f)$ is $q^r/p$.
\item[(ii)] If $p|r$ then the distribution of $N(f)-q^{r/p}$ equals that of $$\sum_{e|r\atop (r/e,p)=1}e\sum_{i=1}^{\pi(e)}X_{e,i}.$$
In particular the mean value of $N(f)$ is $$\frac{q^r}{p}+\lb 1-\frac{1}{p}\rb q^{r/p}.$$
\ei\end{thm}

We also consider what happens when $p,n$ are allowed to vary.

\begin{thm}\label{t2}
Let $p,d$ both tend to infinity and $n,r=1$. Then $N(f)$ converges in distribution to the Poissonian distribution with mean 1, i.e. $\P(m)=e^{-m}/m!$\end{thm}

\begin{thm}\label{t3} Let $p$ be fixed and $n,d\to\infty$. 
\bi\item[(i)] If $(r,p)=1$ or $p>2$ then $$p^{1/2}\lb(1-1/p)r\rb^{-1/2}q^{-r/2}\lb N(f)-\qrp\rb$$ converges in distribution to the 
Gaussian distribution with mean 0 and variance 1.
\item[(ii)] If $p=2$ and $r$ is even then $$2r^{-1/2}q^{-r/2}\lb N(f)-\qrp-\frac{2q^{r/2}}{r}\rb$$
converges in distribution to the Gaussian distribution with mean 0 and variance 1.\ei\end{thm}

\begin{thm}\label{t4} Let $p,d$ both tend to infinity and assume that $n>1$ (not necessarily constant) or $r>1$. Then 
$$p^{1/2}r^{-1/2}q^{-r/2}\lb N(f)-\qrp\rb$$ converges in distribution to the 
Gaussian distribution with mean 0 and variance 1.\end{thm}

The paper is organised as follows: in the next section we review related work dealing with similar problems for other families of L-functions.
In section \ref{back} we provide the necessary background on A-S curves and L-functions. In section \ref{rmm} we describe the random
unitary matrix model for the A-S family of L-functions. In section \ref{pmain} we prove Theorem \ref{thm1} and derive Theorem \ref{thm2} from it.
In section \ref{nonlinear} we consider the 2-level density statistics of the L-zeroes in the A-S family and obtain results which agree with the random unitary matrix model. In section \ref{secdirg} we reformulate our main results in terms of Dirichlet
L-functions over $\fq[x]$ and generalise them to suitable families of Dirichlet L-functions. 
In section \ref{odd} we consider the family of A-S curves defined by \rf{as} with $f$ odd, for which we formulate
conjectures corresponding to a random symplectic matrix model and provide evidence for them in the form of theorem \ref{oddthm}. 
In section \ref{p2} we will discuss the situation
with $p=2$. In section \ref{numpoints} we prove our results on the distribution of the number of points on curves in the A-S family.

\section{Related work}

\subsection{The hyperelliptic ensemble}\label{hyp}

The main inspiration for the present work is the paper \cite{rud}, which studies similar questions and obtains similar
results for an ensemble of hyperelliptic curves. We briefly present the content of that work. 
One considers the family of curves over $\fq$, with $q$ odd, defined by equations of the form
$y^2=f(x)$, with $f$ monic squarefree of degree $d$, with $d$ odd. Denote by $\hd$ the set of all such polynomials $f$.
For $f\in\hd$ the curve defined by $y^2=f(x)$ has an L-function which is a polynomial with integer coefficients, which can be written
as 
$$L_f(z)=\prod_{i=1}^{d-1}(1-\rho_iq^{1/2}z),$$ where $\rho_i$ are the normalised zeroes of $L_f(z)$ satisfying $|\rho_i|=1$.

Denote $T_f^r=\sum_{i=1}^{d-1}\rho_i^r$. It is shown in \cite{rud} that
$$\la T_f^r\ra_{f\in\hd}=\choice{-e_{2,r}+E_{d,r},}{0<r<d,}{E_{d,r},}{r\ge d,}$$ where $e_{2,r}$ is given by \rf{e} and the error term
$E_{d,r}$ satisfies 
$$E_{d,r}=O_q\lb dq^{r/2-d}+dq^{-d/2}\rb$$ for $r\neq d-1$ (if $r=d-1$ an additional summand
of $-q/(q-1)$ appears). This provides evidence in favor of the random symplectic matrix model for the L-functions of hyperelliptic curves
because $$\la\tr U^r\ra_{U\in\usp(d-1)}=\choice{-e_{2,r},}{r<d,}{0,}{r\ge d.}$$

This result is used to obtain a result about the average number of zeroes in short intervals, which again agrees with the random symplectic
matrix model. Namely, let $V$ be as in Theorem \ref{thm2} but with Fourier transform supported in $(-2,2)$ and define $v_d$ as in Theorem
\ref{thm2}. Denote $Z_f=\sum_{j=1}^{d-1}v_d(\th_{f,j})$, where $\rho_{f,j}=e^{i\th_{f,j}}$ are the normalised zeroes of $L_f$ ($\th_{f,j}$
are real numbers well defined modulo $2\pi$) and 
$Z_U=\sum_{j=1}^{d-1}v_d(\th_{U,j})$, where $\rho_{U,i}=e^{i\th_{U,j}}$ are the eigenvalues of a matrix $U\in\usp(d-1)$. 
Then $$\la Z_f\ra_{f\in\hd}=\la Z_U\ra_{U\in\usp(d-1)}+o(1)$$ as $d\to\infty$ and $q$ is fixed.

Note that $\hd$ has on the order of $q^d$ elements while $\#\fd=O\lb q^{d-d/p}\rb$, which means that the A-S family is sparser and we get less averaging. As a result
we get large errors in our estimate for $M_r^d$ already for $r>d(2-2/p)$, while in the hyperelliptic case it is possible to obtain small errors
for $r<2d$.

\subsection{Constant $d$}

Much more is known about the statistics of zeroes for various families of L-functions over finite fields if the degree of the family is held constant
while $q\to\infty$. For an L-function of the form $L(z)=\prod_{i=1}^m(1-q^{1/2}\rho_iz), |\rho_i|=1$ we attach the class of unitary
matrices with eigenvalues $\rho_1,...,\rho_m$. For many families of L-functions (e.g. the hyperelliptic family, the family of Dirichlet
characters and families similar to our A-S family) it was shown by Katz and Sarnak (see \cite{ks}) that as $q\to\infty$ and $m$ is fixed the classes corresponding to the L-functions
of the objects in the family become equidistributed in the set of conjugacy classes of a suitable compact group of matrices
(endowed with the measure induced by the Haar measure on the group), usually
$\U(m),\usp(m)$, the orthogonal group or some similar group, called the symmetry type of the family.
These results cannot be extended to the case $m\to\infty$ (there is no meaning to equidistribution in a varying space)
but the symmetry types observed for constant $m$ can be used to give random matrix models to families of L-functions.
The model for our family of A-S L-functions is presented in section \ref{rmm}.

\subsection{The number of points on curves}

The distribution of the number of points on curves in various families has been studied extensively in recent years.
It follows from \rf{defl},\rf{zeros} below 
that the number of points on a curve $C$ with normalised L-zeroes $\rho_1,...,\rho_{2g}$ ($g$ is the genus of the curve)
is $q+1-q^{1/2}\sum_{i=1}^{2g}\rho_i$. 
The distribution of $T^1_C=\sum_{i=1}^{2g}\rho_i$ as $C$ varies through some family of curves over $\fq$ with genus $g$ as $g\to\infty$
is considered in \cite{kurrud} for the hyperelliptic family, in \cite{kum},\cite{xiong} for the family of trigonal covers of $\mathbf{P}^1$
(with further generalisation to $l$-gonal covers, $l$ a prime dividing $q-1$) and in \cite{plane} for the family of plane curves.
A family of curves in higher dimensional projective spaces has been studied in \cite{kurwig}.
The distribution of the number of points for the family of A-S curves is considered in section \ref{numpoints}.
 
\subsection{Number fields}

Results similar to Theorem \ref{thm2} have previously been obtained for families of L-functions over number fields.
The family of quadratic Dirichlet L-functions 
(with varying modulus) is considered in \cite{miller} and the family of all Dirichlet L-functions with given modulus is considered
in \cite{hughesrud}.

For example let $q$ be a prime number and $\chi$ a Dirichlet character modulo $q$. Let $L(s,\chi)$ be the corresponding L-series and
$\rho_1,\rho_2,...$ its sequence of nontrivial zeroes ordered by increasing absolute value. For simplicity we assume the
Generalised Riemann Hypothesis (although it is not assumed in \cite{hughesrud}) so that $\re\rho_i=1/2$. 
Denote $\gam_i=\im\rho_i$. Denote by $\Xi$ the set of nontrivial Dirichlet characters modulo $q$. 
The average number of L-zeroes satisfying $|\gam_i|<T$ as $\chi$ varies uniformly in $\Xi$, $T$ is fixed and $q\to\infty$ is known to be 
$$\la N(T,\chi)\ra_{\chi\in\Xi}\sim \frac{T}{\pi}\log qT,$$ so we normalise $\del_i=\frac{\log q}{\pi}\gam_i$ 
(now we expect on average one zero with $|\del_i|<1$). Let $V\in\sr$ be a window function,
$Z_\chi=\sum_{i=1}^\ity V(\del_i)$. It is shown in \cite{hughesrud} that if $\hW$ is supported on $\lbb-2,2\rbb$ then
$$\la Z_\chi\ra_{\chi\in\Xi}\to\inti V(t)\dd t$$ as $q\to\ity$. Studying nonlinear statistics they obtain agreement with a random unitary
matrix model (with matrix size around $\log q$) for restricted classes of window functions.

\section{Background on Artin-Schreier curves and L-functions}\label{back}

The material reviewed in this section can be found in \cite{moreno},\cite{rosen},\cite{stepanov}.

\subsection{Notation and conventions}

The notation and conventions introduced in this subsection apply to the entire paper, including the introduction.

When we use the $O$-notation (asymptotic bound) the implicit constant is absolute, except when the bounded quantity depends on a
window function $V$, in which case it may depend on $V$. If there are additional parameters upon which the bound depends we write
them explicitly as a subscript (e.g. $f=O_\eps(g)$). The $o$-notation is always used for $d\to\infty$ and $f=o(g)$ implies that
$f/g\to 0$ as $d\to\infty$ regardless of how the other parameters on which $f,g$ depend vary, except possibly for a window function $V$ which
is always assumed to be fixed. 

For a pair of integers $m,n$ we denote by $(m,n)$ their greatest common divisor. For a pair of polynomials $f,g$, $(f,g)$ denotes their
greatest monic common divisor and $f\bmod g$ denotes the residue class of $f$ modulo $g$. 

For a finite set $S$ we denote by $\#S$ its number of elements.

\subsection{Geometric properties of Artin-Schreier curves}\label{geom}

Let $F$ be a field of characteristic $p>0$.
An Artin-Schreier (A-S shortly) curve over $F$ is the normalisation of the projective closure of the affine curve
defined by an equation of the form $y^p-y=f(x)$ with $f\in F[x]$ a polynomial of degree $d>0$. We denote this curve by $C_f$.
If $(d,p)=1$ then $C_f$ is geometrically irreducible \cite[\S 1.4.2]{stepanov} 
(however this condition is far from necessary for geometric irreducibility).
We assume throughout that $(d,p)=1$.
The affine part of $C_f$ is smooth, as $\d (y^p-y-f(x))/\d y=-1$ never vanishes. The curve $C_f$ has exactly one point (always $F$-rational)
outside its affine part and the genus of $C_f$ is $g=(p-1)(d-1)/2$ (this follows from the material of \cite[\S 4.6.2]{moreno}).

Let $(x,y)$ be an affine point on the curve $C_f$, possibly defined over the algebraic closure of $F$. 
Then $(x,y),(x,y+1),...,(x,y+p-1)$ are all points on $C_f$ and these are all the points
of $C_f$ with abscissa $x$. We see in particular that if $F=\F_q$ is a finite field then the number of points on the affine part of $C_f$ is divisible by $p$ and the total
number of points (including the single point at infinity) equals 1 modulo $p$.

If $g(x)=f(x)+a^p(x^{pk}-x^k)$ for some $a\in F$ then $C_g$ is $F$-isomorphic to $C_f$ via the substitution $x,y\to x,y+ax^k$. If $F$ is a finite
field then every element of $F$ is a $p$-th power and so every curve $C_f$ with $\deg(f)=d$ is isomorphic to a curve $C_g$ s.t. 
$g=a_dx^d+...+a_0$ with $a_{kp}=0$ for all $k>0$. 

Now assume that $F=\fq$ is a finite field. An element $a\in F$ can be written as $a=b^p-b,b\in F$ iff $\tr_{F/\F_p}a=0$ (this follows
from the Hilbert 90 theorem or more simply by noting that the map $b\mapsto b^p-b$ is linear with one-dimensional kernel $\F_p$, while
its image is contained in the kernel of the trace map). If $g(x)=f(x)+a$ with $a\in F$ satisfying $\tr_{F/\F_p}a=0$ then $C_g$ is isomorphic
to $C_f$ via $x,y\mapsto x,y+b$, where $a=b^p-b$. Even if $\tr_{F/\F_p}a\neq 0$ the curves $C_g$ becomes isomorphic to $C_f$ over $\F_{q^p}$
(because $\tr_{\F_{q^p}/\F_p}a=0$). If $g=f+a$ we say that $C_g$ is a twist of $C_f$.

\subsection{Artin-Schreier curves over finite fields and their L-functions}\label{lfunc}

Let $F=\F_q,q=p^n$ be a finite field. Let $C$ be a smooth projective curve over $\fq$ with genus $g$. Denote by $N_r(C)$ the number of $\F_{q^r}$-points
on $C$. The L-function of $C$ is defined by the power series 
\beq\label{defl}L(z)=\exp\lb \sum_{r=1}^\infty\frac{q^r-1-N_r(C)}{r}z^r\rb.\eeq
It turns out that $L(z)$ is a polynomial of degree $2g$ and in fact we may write
\beq\label{zeros}L(z)=\prod_{i=1}^{2g}(1-\rho_iq^{1/2}z).\eeq The normalised zeroes $\rho_i$ come in conjugate pairs and they all satisfy $|\rho_i|=1$
(the Riemann Hypothesis for curves over a finite field). For all these properties of the L-function of a curve over a finite field see
\cite[\S 5]{rosen}, \cite[\S V]{stepanov}, \cite[\S 3]{moreno}.

Let $\zeta$ be a primitive $p$-th (complex) root of unity. We define an additive character $\psi:\F_p\to\C^\times$ by
$\psi(a)=\zeta^a$ (this is well defined). All the nontrivial additive characters of $\F_p$ are of this form and there are $p-1$ characters
corresponding to the $p-1$ roots of unity. We denote by $\t_{s/t}:\F_s\to\F_t$ the trace map from the field with $s$ elements to the field with $t$
elements, provided $s$ is a power of $t$.

Now let $f\in F[x]$ be a nonconstant polynomial of degree $d$ (we always assume $(d,p)=1$) and $C_f$ the corresponding A-S curve. Let $r$ be a natural number and $x\in\F_{q^r}$.
Any element $a\in\fqr$ can be written as $a=y^p-y$ with $y\in\fqr$ iff $\t_{q^r/p}a=0$ (see previous subsection).
Applying this to $f(x)$ we see that $C_f$ has an affine point with abscissa $x$ iff $\t_{q^r/p}f(x)=0$, in which case it has exactly $p$ such
points (namely $(x,y),(x+1,y),...,(x+p-1,y)$).
Using the orthogonality relation for the additive characters of $\F_p$ this can be restated as follows: the number of $\fqr$-points on $C_f$ with abscissa $x\in\fqr$ equals 
$$1+\sum_{\psi\in\Psi}\psi(\t_{q^r/p}x)$$ and so the total number of points on $C_f$ (including the infinite point) is
\beq\label{numpts}N_r(C_f)=q^r+1+\sum_{\al\in\fqr}\sum_{\psi\in\Psi}(\t_{q^r/p}f(\al)).\eeq

Define $$L_{f,\psi}(z)=\exp\lb\sum_{r=1}^\infty\sum_{\al\in\fqr}\psi(\t_{q^r/p}f(\al))\frac{z^r}{r}\rb.$$
It follows from \rf{defl} and \rf{numpts} that the L-function of $C_f$ can be written as a product
$$L_f(z)=\prod_{\psi\in\Psi}L_{f,\psi}(z).$$
Each function $L_{f,\psi}$ turns out to be a polynomial of degree $d-1$ with constant term $1$, see \cite[\S I.3]{stepanov}, \cite[\S 4.6.2]{moreno}.
We can write $$L_{f,\psi}=\prod_{i=1}^{d-1}(1-\rho_{\psi,i}q^{1/2}z),$$ where the $\rho_{\psi,i},1\le i\le d-1,\psi\in\Psi$ are
all the normalised zeroes of $L_f(z)$ and they satisfy $|\rho_{\psi,i}|=1$.  
To understand the behaviour of the zeroes of $L_f(z)$ it is enough to study the zeroes of the individual $L_{f,\psi}$ and the relationship
between the L-functions corresponding to different characters. These functions are called primitive L-functions or the primitive factors of $L_f$.
Note that for two conjugate characters $\psi,\pb\in\Psi$ the two L-functions $L_{f,\psi},L_{f,\pb}$ are conjugate and so are their zeroes.
Thus for $p>2$ the primitive factors come in conjugate pairs (for $p=2$ there is only one nontrivial character and $L_f$ itself is a primitive
L-function).

Let $f,g\in\fq[x]$ be polynomials of degree $d$ s.t. $g(x)=f(x)+a(x^{kp}-x^k)$ for some $a\in F$ and natural $k$. Then for all $\al\in\fqr$
we have $\tqrp g(\al)=\tqrp f(\al)$, since $$\tqrp(a(\al^{kp}-\al^k))=\tqp(a\tqrq(\al^{kp}-\al^k))=\tqp(0)=0.$$
Therefore $L_{f,\psi}=L_{g,\psi}$. This agrees with the fact that $C_f,C_g$ are isomorphic and so $L_f=L_g$.
Now assume that $g(x)=f(x)+a$ for some $a\in\fq$. From \rf{defl} we see that $L_g(z)=L_f(\psi(\tqp a)z)$ and denoting
by $\rho_i,\rho'_i,1\le i\le d-1$ the normalised roots of $L_{f,\psi},L_{g,\psi}$ respectively we see that $\rho'_i=\psi(-\tqp a)\rho_i$.
If $\tqp a=0$ then $L_{f,\psi}=L_{g,\psi}$ and as we have seen in the previous subsection the curves $C_f,C_g$ are isomorphic in this case.

\subsection{The Artin-Schreier family}\label{asfam}

Denote $$\fd=\left\{f(x)=\sum_{i=0}^d a_ix^i\big|a_i\in\fq, a_d\neq 0, a_i=0 \mbox{ if } p|i \mbox{ and } i\ge 0\right\}.$$ We refer to $\{C_f\}_{f\in\fd}$ as the Artin-Schreier (A-S in short) family of curves with parameter
$d$ over $\fq$ and to $\{L_{\psi,f}(z)\}_{f\in\fd}$ the as the A-S family of L-functions with parameter $d$ over $\fq$. Note that the latter does
not depend on the choice of $\psi\in\Psi$ because for $a\in\fp^\times$ we have $L_{f,\psi^a}(z)=L_{af,\psi}(z)$ (this follows from \rf{defl}),
so replacing $\psi$ with $\psi^a$ permutes the family of L-functions.

Fix some element $c\in\fq$ with $\tqp c=1$.
For $w\in\fp$ denote $\fd^w=\{f+wc|f\in\fd\}$. By the observations in the end of the previous subsection we have 
$$\{L_{f,\psi}(z)|f\in\fd^w\}=\{L_{f,\psi}(\psi(w)z)|f\in\fd\},$$ so the statistics of zeroes of the L-functions of $f\in\fd^w$ is essentially
the same as that of $f\in\fd$.

The size of the family $\fd$ and each $\fd^w$ is $\#\fd^w=\#\fd=q^{d-\floor{d/p}}(q-1)$. 
Denote by $\fd'$ the set of all degree $d$ polynomials in $\F_q[x]$.
Define the map $\mu:\fd'\to\cup_{w\in\fp}\fd^w$ by 
$$\mu\lb \sum_{i=0}^da_ix^i\rb=(\tqp a_0)c+\sum_{i=1}^d\lb \sum_{j=1}^{\floor{\log_p(d/i)}}a_{ip^j}\rb x^i.$$
This map is precisely $q^{\floor{d/p}+1}/p$ to one. 
It follows from the observations in the end of the previous subsection that for any $r$ and any $\al\in\fqr$ we have
$\tqrp(\mu(f)(\al))=\tqrp f(\al)$ and therefore $L(f,\psi)=L(\mu(f),\psi)$.

We conclude that studying the statistics of zeroes of L-functions of $f\in\fd'$ reduces to studying it for $\fd^w,w\in\F_p$, which
in turn reduces to studying it for the family $\fd$. Henceforth we only consider the family $\fd$.

\section{The random matrix model}\label{rmm}

In recent decades it has been suggested that the zeroes of L-functions (of all sorts) behave as the spectra of matrices drawn randomly (with some
natural measure) from some classical ensemble of matrices. We will illustrate this approach on our example of the A-S L-functions.
For the rest of this section assume $p>2$.

Denote by $\U_N$ the group of $N\times N$ unitary matrices. This is a compact Lie group and so it has a Haar measure.
We may draw a random matrix $U\in\U_N$ uniformly w.r.t. the Haar measure and ask about the statistics of its spectrum.
The eigenvalues of $U$ lie on the unit circle, just like the zeroes of $L_{f,\psi}$. We take $N=d-1$ (recall $d=\deg f$) and
we model the set of roots of $L_{f,\psi}$ for fixed $\psi$ and "random" $f$ (i.e. $f$ may vary through some large family, e.g. $\fd$ with
either $q\to\infty$ or $d\to\infty$, or both) by the spectrum of a random matrix from $\U_{d-1}$.
This model is suggested by the result due to N. Katz and P. Sarnak for the case of fixed $d$ and $q\to\infty$, stating that the sets of normalised zeroes
of the L-functions in the A-S family become equidistributed in the space of sets of eigenvalues of matrices in $\U_{d-1}$ with the measure induced from the Haar measure on $U_{d-1}$. See Theorem 3.9.2 in \cite{monper}.
To model the $p-1$ sets of zeroes of $L_{f,\psi},\psi\in\Psi$ jointly for $p>2$ we may take $(p-1)/2$ independent random matrices from
$\U_{d-1}$ and
their conjugates (see section \ref{lfunc}).

Now we formulate some conjectures on the statistics of the zeroes of $L_{f,\psi},f\in\fd$, of which our main theorems are special cases.
In all that follows assume $(d,p)=1$.
We are interested in the case where $d\to\infty$ and $q$ may be fixed or vary as we please.
First we consider the linear statistics - the number of zeroes in short intervals and the related quantity of the average of powers of the zeroes.
Since multiplying a matrix $U\in\U_N$ by a scalar matrix $e^{i\th}I_N$ rotates the eigenvalues by an angle of $\th$ it is obvious that
the average number of eigenvalues contained in an arc of length $l$ is $Nl/2\pi$ (as $U$ is drawn uniformly at random w.r.t. the Haar measure).
We will be interested in $l=O(1/d)$ as $d\to\infty$, the so-called local regime (it is easier to obtain results for larger arcs).
Using our model we formulate

\begin{conj}\label{conj1} Take any $\psi\in\Psi$. 
Let $C>0$ be a constant. For every natural number $d$ let $I_d$ be any arc on the unit circle of length
$C/d$. Then the average number of zeroes of $L_{f,\psi}$ contained in $I_d$ as $f$ is chosen uniformly at random from
$\fd$ is $C/2\pi+o(1)$ as $d\to\infty$.\end{conj}

Instead of just looking at arcs we may take a smooth window function to count the zeroes. Let $V(t)\in\sr$. 
The function 
$v_d(t)=\sum_{n=-\infty}^\infty V(d(t+2\pi n))$ is well-defined and periodic with period $2\pi$. It can be viewed as a function
on the unit circle. We say that $v_d(t)$ is the periodic window function associated with $V(t)$ with scaling parameter $d$.
For every $z$ with $|z|=1$ and real number $\th$ the value of $v_d(\arg(z)+\th)$ is well-defined.
Denote $S_f=\sum_{j=1}^{d-1}v_d(\th_j-\th)$, where $\rho_i$ are the normalised zeroes of $L_{f,\psi}$. Conjecture \ref{conj1} is equivalent
to the following statement: the average of $S_f$ as $f$ is chosen uniformly at random from
$\fd$ is $$\frac{1}{2\pi}\int_{-\infty}^\infty V(t)dt+o(1)$$ as $d\to\infty$ for any $V\in\sr$ and $\th\in\R$ (because the indicator of an interval can be approximated by a window function 
in $\sr$ and any
window function can be approximated by a superposition of interval indicators). 

Now we consider the quantity $T_{f,\psi}^r=\sum_{i=1}^{d-1}\rho^r_{f,\psi,i}$, where as usual $\rho_{f,\psi,i}$ are the normalised zeroes
of $L_{f,\psi}$. The uniform distribution of the L-zeroes on the unit circle suggests the following

\begin{conj}\label{conj3} Take any $\psi\in\Psi$. For every $\eps>0$ the average of $T_{f,\psi}^r$ where $f$ is drawn uniformly at random from $\fd$
is $$O_\eps\lb q^{\eps r+(1/p-1/2)d}\rb$$ as $d\to\infty$ and $r\ge d$ ($q$ may vary with $d$ as we please).\end{conj}

It can be shown by a standard argument that Conjecture \ref{conj3} combined with Theorem \ref{thm1} 
implies Conjecture \ref{conj1}. See the proof of Theorem
\ref{thm2} in section \ref{pthm2} for this kind of argument. We remark that Conjecture \ref{conj3} would follow from a function field analogue
of a conjecture of H. Montgomery about the distribution of primes in arithmetic progressions (see \cite[\S 13.1]{mv}).

At this point a simpler model for the zeroes of L-functions in the A-S family would be just $d-1$ independent random points on the unit circle
(with uniform distribution), which is also consistent with Conjectures \ref{conj1},\ref{conj3} and Theorems \ref{thm1},\ref{thm2}.
However in section \ref{nonlinear} we study nonlinear statistics of the zeroes which show agreement with the random unitary matrix model and
disagreement with the independent random points model.

Finally we note that for $p=2$ we need a different model, namely a random symplectic matrix model. 
See section \ref{p2} for a description of this model and some partial results.

\section{Proof of the main results}\label{pmain}

\subsection{Proof of Theorem \ref{thm1}}\label{pthm1}

We keep the notation of the previous section. Let $p,q$ be as in section \ref{intro}, $\psi\in\Psi$, $d$ a natural number satisfying $(d,p)=1$.
First we need a lemma

\begin{lem}\label{lem1} Let $\psi\in\Psi$ be a character, $f\in\F_q[x]$ of degree $d$. Let $\rho_i,1\le i\le d-1$ be the normalised
zeroes of $L_{f,\psi}$ and $T^r_{f,\psi}=\sum_{i=1}^{d-1}\rho_i^r$. Then $$T^r_{f,\psi}=-q^{-r/2}\sum_{\al\in\fqr}\psi(\tqrp f(\al)).$$
\end{lem}

\begin{proof} This is a well known fact that follows directly from \rf{defl} and \rf{zeros}. See \cite[\S I.3.3]{stepanov} for details.\end{proof}

For the rest of this section fix $\psi\in\Psi$.
To prove Theorem 1 we need to estimate the average of the sum $\sum_{\al\in\fqr}\psi(\tqrp f(\al))$ as $f$ varies uniformly through $\fd$.
Recall that $\fd$ consists of the degree $d$ polynomials $f=\sum_{i=0}^d a_ix^i$ with
$a_d\neq 0$ and $a_{kp}=0$ for all $k\ge 0$.
 
We begin with a simple observation that establishes a weak form of Theorem 1, namely with $r<d$.

\begin{lem}\label{lem0} Assume $r<d$. Let $\al\in\fqr$ be an element. Then 
\beq\label{elem0}\afd{\psi(\tqrp f(\al))}=\choice{1,}{\al=0 \mbox{ or } p|r,\al\in\F_{q^{r/p}},}{0,}{\mbox{otherwise}.}\eeq\end{lem}

\begin{proof} If $\al=0$ the assertion is clear since $f(\al)=f(0)=0$ and so\\ $\psi(\tqrp f(\al))=1$ for all $f\in\fd$.
If $p|r$ and $\al\in\F_{q^{r/p}}$ then for all $f\in\fd$ we have $f(\al)\in\F_{q^{r/p}}$, so $\tqrp f(\al)=p\cdot\t_{q^{r/p}/p}f(\al)=0$ and 
$\psi(\tqrp f(\al))=1$.
 
Now assume that $\al\neq 0$ and if $p|r$ then $\al\not\in\F_{q^{r/p}}$. This means that the minimal polynomial $h$ of $\al$ over $\fq$ satisfies
$(r/{\deg h},p)=1$.
Denote $\fd''=\{f\in\fq[x]|\deg f=d,f(0)=0\}$.
Recall the definition of the map $\mu$ in section \ref{asfam}.
Since the $\mu|_{\fd''}:\fd''\to\fd$ is precisely $q^\floor{d/p}$ to one and preserves $\tqrp f(\al)$ so we may replace $\fd$ by $\fd''$
in \rf{elem0}.
Let $h\in\F_q[x]$ be the minimal polynomial of $\al$ over $\F_q$. Since $d>r$ the map $\pi:\fd''\to\F_q[x]/h\cong \fqr$ defined by
$\pi(f)=(f/x)\bmod h$ is exactly $(q-1)q^{d-r-1}$ to one and since $x$ is invertible modulo $h$ (as $\al\neq 0$) so is the map
$\pi':\fd''\to\fqr\cong\fq[x]/h$ defined by $\pi'(f)=f(\al)$. Thus each value of $f(\al)\in\F_{q^{\deg h}}\cong\fq[x]/h$ is obtained equally many times as $f$ ranges
through $\fd''$. Since $(r/{\deg h},p)=1$ the value of $\psi(\tqrp\gam)=\psi(\t_{q^{\deg h}/p}\gam)^{r/{\deg h}}$ is uniformly distributed among the $p$-th roots of unity as
$\gam$ ranges through $\F_{q^{\deg h}}$, which proves the assertion of the lemma.\end{proof}

The following corollary establishes Theorem \ref{thm1} for $r<d$.
\begin{cor}\label{cor1}Assume $r<d$. Then $M_d^r=-e_{p,r}q^{r/p-r/2}+(e_{p,r}-1)q^{-r/2}.$\end{cor}
\begin{proof} By Lemma \ref{lem1} we have \beq\label{erld}M_d^r=\afd{T^r_{f,\psi}}=-q^{-r/2}\sum_{\al\in\fqr}\afd{\psi(\tqrp f(\al))}.\eeq
Now using Lemma \ref{lem0} we see that the RHS of \rf{erld} equals $q^{-r/2}$ if $(r,p)=1$ (only $\al=0$ contributes 1 to the sum)
or $q^{r/p-r/2}$ if $p|r$ (each $\al\in\fqr$ contributes 1 to the sum).\end{proof}

To go further we need some lemmata.

\begin{lem}\label{lem2} Let $\al\in\fqr$ have monic minimal polynomial $h(x)=\sum_{i=0}^rc_rx^r\in\F_q[x]$ of degree $r$. 
 Assume that
for some $0<k\le r$ with $(k,p)=1$ we have $c_{r-k}\neq 0$. Then for the minimal such $k$ we have $$\t_{q^r/q}(\al^k)=-kc_{r-k}\neq 0$$
and for all $0<j<k$ we have $$\sum_{i=1}^r\al_i^j=0.$$\end{lem}

\begin{proof} Denote by $\al_1=\al,\al_2,...,\al_r$ the conjugates of $\al$ over $\F_q$. Denote by $\sig_i(x_1,...,x_r)$ the degree $i$ elementary symmetric polynomial in $r$ variables. We have 
$c_{r-i}=(-1)^i\sig_i(\al_1,...,\al_r)$ for all $1\le i\le r$. Denote $s_i(x_1,...,x_r)=\sum_{j=1}^rx_j^i$.
Newton's identity (see \cite[\S 3.1.1]{prasolov}) states that for all $1\le m\le r$ we have
\beq\label{newton}m\sig_m=(-1)^{m+1}s_m-\sum_{i=1}^{m-1} s_i\sig_{m-i}.\eeq
We show by induction that for $1\le i<k$ we have $s_i(\al_1,...,\al_r)=0$. If $k>1$ then the case $i=1$ is clear as by assumption
$\sig_1(\al_1,...,\al_i)=-c_{r-1}=0$. Assume that $i<k$ and that $s_j(\al_1,...,\al_r)=0$ holds for all $1\le j<i$. 
By \rf{newton} we have $$ic_{r-i}=\pm i\sig(\al_1,...,\al_r)=\pm s_i(\al_1,...,\al_r)$$ (the other terms in the identity
are zero by the induction hypothesis). Now if $i$ is not divisible by $p$ the assumption on $k$ implies that $c_{r-i}=0$ and if $i$
is divisible by $p$ we still have $ic_{r-i}=0$ (as we are in characteristic $p$). This completes the induction.
Now again we see from \rf{newton} that $$s_k(\al_1,...,\al_r)=(-1)^{k+1}k\sig_k(\al_1,...,\al_r)=-kc_{r-k}$$ as required.\end{proof}

\begin{lem}\label{mainlem} Let $\al\in\fqr$ have monic minimal polynomial $h(x)=\sum_{i=0}^rc_rx^r\in\F_q[x]$ of degree $r$. Assume that
either $r<d$ or $r\ge d$ and for some $0\le k<d$ with $(k,p)=1$ we have $c_{r-k}\neq 0$. Then 
$\afd{\psi(\tr_{q^r/p}f(\al))}=1$.
If $d\ge r$ and there is no such $k$ then if $c_{r-d}=0$ we have 
$\afd{\psi(\tr_{q^r/p}f(\al))}=1$ and if $c_{r-d}\neq 0$ we have
$\afd{\psi(\tr_{q^r/p}f(\al))}=-1/(q-1)$.\end{lem}

\begin{proof}
The case $r\le d$ follows from Lemma \ref{lem0}, so we assume $d<r$.
First assume there exists $0\le k<d$ s.t. $(k,p)=1$ and $c_{r-k}\neq 0$. Let $k$ be minimal with this property.
By the previous lemma $\t_{q^r/q}(\al^k)=-kc_{r-k}\neq 0$. Therefore there exists $a\in\F_q$ s.t. 
$$\t_{q^r/p}(a\al^k)=\t_{q^r/q}(a\t_{q/p}(\al^k))\neq 0.$$ Now the set $\fd$ can be partitioned into subsets of the form
$$S_g=\{g,g+ax^k,g+2ax^k,...,g+(p-1)ax^k\}.$$ Note however that
\begin{multline*}
\sum_{f\in S_g}\psi(\tr_{q^r/p}f(\al)=\psi(\tr_{q^r/p}g(\al))\sum_{i=0}^{p-1}\psi(\tr_{q^r/p}(ia\al^k))=\\
=\psi(\tr_{q^r/p}g(\al))\sum_{i=1}^{p-1}\psi(\tr_{q^r/p}(a\al^k))^i=0\end{multline*}
because $\psi(\tr_{q^r/p}(a\al^k))$ is a primitive $p$-th root of unity. Since $\fd$ is partitioned into sets of the form $S_g$ we get
the first claim of the lemma.

Now assume that for all $0\le k<d,(k,p)=1$ we have $c_{r-k}=0$. Take some $\al\in\fqr$ of degree $r$.
By the previous lemma we get that $\tr_{q^r/q}(\al^i)=0$ for $0\le i< d$ and so $\tqrp(a\al^i)=0$ for all $a\in\F_q$ and by the second part of the lemma we have $\tqrq(\al^d)=-dc_{r-d}$ and so
$\tqrp(a\al^d)=-d\tqp(ac_{r-d})$ for every $a\in\F_q$. Thus for every $f=\sum_{i=0}^d\in\fd$ we have $\tqrp f(\al)=-d\tqp(a_dc_{r-d})$.

If $c_{r-d}=0$ then $\afd{\psi(\tqrp f(\al))}=1$. Assume $c_{r-d}\neq 0$.
The leading coefficient of $f\in\fd$ is distributed uniformly in $\fqs=\F_q\sm\{0\}$ and so is $a_dc_{r-d}$. We have
$$\la\psi(\tqp a)\ra_{a\in\fqs}=-1/(q-1)$$ because as $a$ ranges through $\fqs$ each nonzero value of $\tqp a$ occurs $q/p$
and zero occurs $q/p-1$ times, so $\sum_{a\in\fqs}\psi(a)=-1$. This concludes the proof of the lemma.
\end{proof}

For $r>d$ denote by $\eta_d(r)$ the number of monic irreducible polynomials
$h(x)=x^r+\sum_{i=0}^{r-1}c_ix^i$ s.t. $c_{r-k}=0$ for all $1\le k< d$ with $(k,p)=1$. Denote by $\eta^0_d(r)$ the number of such
polynomials with $c_{r-d}=0$ (if $r=d$ we define $\eta^0_d(r)=0$).

\begin{prop}\label{irr} 
$$M^r_d=\frac{q^{-r/2+1}}{q-1}\sum_{s|r, (r/s,p)=1, s\ge d}s(\eta_d(s)/q-\eta^0_d(s))-e_{p,r}q^{r/p-r/2}+(e_{p,r}-1)q^{-r/2}.$$\end{prop}

\begin{proof} By Lemma \ref{lem1} we have
\beq\label{irr1}M^r_d=q^{-r/2}\sum_{\al\in\fqr}\afd{\psi(\tqrp f(\al))}.\eeq
Every $\al\in\fqr$ has degree $s|r$ over $\F_q$. First let $s|r$ be such that $(r/s,p)=1$. For $\al$ of degree $s$ (over $\F_q$) and $f\in\fd$
we have $\tqrp f(\al)=(r/s)\t_{q^s/p}f(\al)$ and so $$\psi(\tqrp f(\al))=\psi^{r/s}(\t_{q^s/p}f(\al)).$$ 
By Lemma \ref{mainlem} applied to $s,\psi^{r/s}$ instead of $r,\psi$ we see that the contribution
of all $\al\neq 0$ of degree $s$ to the RHS of \rf{irr1} is $$\frac{q^{-r/2+1}}{q-1}s(\eta_d(s)/q-\eta^0_d(s)),$$
since each irreducible polynomial $h=\sum c_ix^i$ of degree $s>d$ has $s$ roots, each contributing 1 to the sum if $c_{s-d}=0$ and
$-1/(q-1)$ otherwise (elements of degree $s\le d$ contribute nothing by Lemma \ref{lem0}).
If $(r,p)=1$ we obtain the assertion of the lemma. It remains to consider the contribution of $\al=0$ and $\al\in\F_{q^{r/p}}$ in case that $p|r$.
Since for $\al=0$ and $\al\in\F_{q^{r/p}}$ we have $\tqrp f(\al)=0$, this contribution is obviously $-e_{p,r}q^{r/p-r/2}+(e_{p,r}-1)q^{-r/2}$.\end{proof}

Theorem 1 follows at once from proposition \ref{irr}. Indeed the total number of polynomials of the form $h(x)=x^s+\sum_{i=0}^{s-1}c_ix^i$
with $s\ge d,s|r$ and $c_{s-kp}=0,1\le k\le\floor{d/p}$ is at most $O(q^{r-d+\floor{d/p}})$ and therefore 
$$M^r_{d,\psi}=-e_{p,r}q^{r/p-r/2}+O\lb rq^{r/2-(1-1/p)d}+q^{-r/2}\rb.$$ 
Note that the assertion of the theorem is interesting only for $r<2d(1-1/p)$, because
by the definition of $T^f_{d,\psi}$ (see Lemma \ref{lem1}) we have $T^f_{d,\psi}=O(d)$ for $f\in\fd$ and so $M^r_{d,\psi}=O(d)$.

\subsection{Proof of Theorem \ref{thm2}}\label{pthm2}

Now we deduce Theorem \ref{thm2} from Theorem \ref{thm1}. For this subsection we assume $p>2$.
Let $V\in\sr$ be a fixed window function.
From a window function $W$ we may construct a periodic window function with parameter $d$ (natural number)
as follows:
$$v_d(t)=\sum_{r=-\infty}^\infty V(d(t+2\pi r)).$$
We also define $v_{d,\th}(t)=v_{d}(t-\th)$.
The function $v_{d,\th}$ has period $2\pi$ and as $d\to\infty$ it becomes "localised" at points of the form $\th+2\pi m,m\in\Z$.
The Fourier transform of $V(t)$ is given by
$$\hat{V}(s)=\frac{1}{2\pi}\inti V(t)e^{-ist}\dd t$$ and the $r$-th Fourier coefficient of $v_{d,\th}$ is
$$\hat{v}_{d,\th}(r)=\int_0^{2\pi}v_{d,\th}(t)e^{-irt}\dd t.$$
A simple calculation shows that \beq\label{fourlem}\hat{v}_{d,\th}(r)=\frac{e^{-ir\th}}{d}\hat{V}\lb\frac{r}{d}\rb.\eeq

\begin{lem} For $f\in\fd$ denote $$S_f=\sum_{j=1}^{d-1}v_{d,\th}(\th_j),$$ where
$\rho_j=e^{i\th_j}$ are the normalised zeroes of $L_{f,\psi}$ ($\th_j$ are real numbers well defined modulo $2\pi$). Then
$$S_f=\hw(0)+\sum_{r=1}^\infty\lb\hw\lb\frac{r}{d}\rb\frac{e^{-ir\th}}{d}T_{f,\psi}+\hw\lb-\frac{r}{d}\rb\frac{e^{ir\th}}{ d}\overline{T^r_{f,\psi}}\rb.$$\end{lem}

\begin{proof} Since $V\in\sr$ the function $v_{d,\th}$ is smooth and so for $|z|=1$ we have
$$v_{d,\th}(\arg z)=\sum_{r=-\infty}^\infty\hwdt(r) z^r.$$ Applying this to $z=\rho_j$, noting that
$\rho_j^{-r}=\bar{\rho}_j^r$, using \rf{fourlem} and summing over $i$ we obtain the assertion
of the lemma.\end{proof}

\begin{cor}$$\afd{S_f}=\hw(0)+\sum_{r=1}^\infty\lb\hw\lb\frac{r}{d}\rb e^{-ir\th}+\hw\lb-\frac{r}{d}\rb e^{ir\th}\rb\frac{M^r_d}{d}.$$\end{cor}

\begin{proof} Just average the previous lemma over $f\in\fd$ and note that $M_d^r\in\R$ by Proposition \ref{irr}.\end{proof}

Now we are ready to prove Theorem \ref{thm2}. Assume that $\hw$ is supported in $\lb-(2-2/p),2-2/p\rb$. There exists $\eps>0$ s.t.
$\hw(r/d)=0$ for all $r\ge (2-2/p-\eps)d$. Using
the last corollary and Theorem \ref{thm1} we obtain
\begin{multline*}\afd{S_f}=\hw(0)+\sum_{r=1}^\floor{(2-2/p-\eps)d}\lb\hw\lb\frac{r}{d}\rb e^{-ir\th}+\hw\lb-\frac{r}{d}\rb e^{ir\th}\rb\frac{M_d^r}{d}=\\=
\hw(0)+\sum_{r=1}^\floor{(2-2/p-\eps)d}O\lb q^{r/p-r/2}+rq^{r/2-(1-1/p)d}\rb\frac{1}{d}=\\=
\hw(0)+O(1/d)+O\lb\frac{q^{-\eps d}}{d}\rb=\hw(0)+o(1)
\end{multline*}
as $d\to\infty$ (note that we used the fact that $\hw=O(1)$ since $V\in\sr$ is fixed). 
It is now enough to notice that $\hw(0)=\frac{1}{2\pi}\inti V(t)\dd t$.

{\bf Remark.} As can be seen from the above proof, the $o(1)$ term in Theorem \ref{thm2} can be replaced with $O(1/d)$. 

\subsection{p=2}\label{p2}

If $p=2$ the assertion of Theorem \ref{thm2} does not hold and has to be modified. The random matrix model needs to be modified as well.
Indeed there is only one character $\psi\in\Psi$ and the L-function $L_f=L_{f,\psi}$ for $f\in\fd$ is already primitive. Since it has real coefficients it is obvious that the set of normalised zeroes consists of conjugate pairs and cannot be "random" in the space of eigenvalue sets of unitary matrices.
Instead we should consider a random unitary symplectic matrix $U\in\usp(d-1)$ (thus we denote the group of $(d-1)\times(d-1)$ unitary symplectic
matrices). In fact for $p=2$ the curves in the A-S family are hyperelliptic and the usual model for families of hyperelliptic curves is the
random symplectic matrix model, see for example the work mentioned in section \ref{hyp}.

For the unitary symplectic group the following holds:
$$\la T_U^r\ra_{U\in\usp(d-1)}=\choice{-e_{2,r},}{r<d,}{0,}{r\ge d,}$$
see \cite[\S 4]{ds}. From this and Theorem \ref{thm1} we conclude that for $r<d$ we have
$$M_d^r=\la T_U^r\ra_{U\in\usp(d-1)}+O\lb q^{-r/2}\rb.$$
From this one can derive using the method of section \ref{pthm2} the following

\begin{cor} Let $V\in\sr$ be a window function, $$v_d(t)=\sum_{n=-\ity}^\ity V(d(t+2\pi n)).$$
For $f\in\od$ denote $Z_f=\sum_{j=1}^{d-1} v_d(\th_j)$, where $\rho_j=e^{i\th_j}$ are the
normalised zeroes of $L_f$. Similarly for a matrix $U\in\usp(d-1)$ with eigenvalues $\rho_j=e^{i\th_j}$ denote $Z_U=\sum_{j=1}^{d-1} v_d(\th_j)$.
Assume that $\hw$ is supported on $\lbb -1,1\rbb$. Then
$$\afd{Z_f^r}\to\la Z_U^r\ra_{U\in\usp(d-1)}$$ as $d\to\ity$.\end{cor}

\section{Nonlinear statistics}\label{nonlinear}

Theorems \ref{thm1} and \ref{thm2} suggest that the zeroes of a random L-function from the A-S family are (at least on average)
rather uniformly distributed on the unit circle, but this seems like a weak confirmation of the random unitary matrix model.
A simpler model would be $d-1$ independent random points on the circle (with uniform distribution).
In this section we study more delicate statistics of the zeroes and show agreement with the random unitary matrix model and disagreement
with the independent random points model.

We preserve the notation of the previous sections. Let $V(t,u)\in\srr$ be a two-variable window function, $N$ a natural number and $v_N(t,u)$ the periodic window function associated with $V$
by $$v_N(t,u)=\sum_{m,n=-\ity}^\ity V(N(t+2\pi m),N(u+2\pi n)).$$ Let $\rho_1,...,\rho_N$ be $n$ points on the unit circle, $\rho_j=e^{i\th_j}$. 
Finally let $\th$ be some fixed real number. We consider the
2-level density function (at $\th$):
$$S^2_\th(\rho_1,...,\rho_N)=\sum_{1\le j,k\le N\atop{j\neq k}}v_N(\th_j-\th,\th_k-\th).$$
For a matrix $U\in\U(N)$ with eigenvalues $\rho_1,...,\rho_N$ we denote $$S^2_\th(U)=S^2_\th(\rho_1,...,\rho_N)$$ and for an A-S L-function $L_{f,\psi}$ with
normalised zeroes $\rho_1,...,\rho_{d-1}$ we denote $$S^2_\th(f,\psi)=S^2_\th(\rho_1,..,\rho_{d-1}).$$

It is easy to see that if $\rho_1,...,\rho_N$ are selected uniformly and independently on the unit circle, then the average of $S^2_\th(\rho_1,...,\rho_N)$
tends to $$\frac{1}{4\pi^2}\inti\inti V(t,u)\dd t\dd u$$ as $N\to\ity$ (for any $\th$). On the other hand we have the following result (see \cite[\S AD.2]{ks}):
\beq\label{upair}\la S^2_\th(U)\ra_{U\in\U(N)}\to\frac{1}{4\pi^2}\inti\inti V(t,u)\lb 1-\lb\frac{\sin(\pi(t-u))}{\pi(t-u)}\rb^2\rb \dd t\dd u.\eeq
as $N\to\infty$ (the average is taken w.r.t. the Haar measure), for any $\th$.
Theorem \ref{thm2p}, which we prove in the present section, provides evidence for the random unitary matrix model.

\subsection{Product of traces}

For the rest of section \ref{nonlinear} we assume that $p>2$.
Just as we used the quantities $T^r_{f,\psi},M^r_d$ to study the linear statistics of the zeroes of $L_{f,\psi}$ we introduce the
quantities $$T^{r,s}_{f,\psi}=T^r_{f,\psi}T^s_{f,\psi}, M^{r,s}_d=\afd{T^{r,s}_{f,\psi}}$$
(again $M^{r,s}_d$ does not depend on the choice of $\psi\in\Psi$).
We also define $T^r_{f,\psi}$ for any integer $r$ (possibly negative) by the same expression $T^r_{f,\psi}=\sum_{j=1}^{d-1}\rho_j^r$
(where $\rho_j$ are the normalised zeroes of $L_{f,\psi}$) and extend the definition of $M^r_d,T^{r,s}_{f,\psi},M^{r,s}_d$
to all integers $r,s$. Note that $T^{-r}_{f,\psi}=\ol{T^r_{f,\psi}}$.
Good estimates for these quantities provide good estimates for the quadratic statistics of the L-zeroes, such as the square of the number
of points in short intervals and the 2-level density.

\begin{lem}\label{trs} For $r,s>0$ we have $$T^{r,s}_{f,\psi}=q^{-(r+s)/2}\sum_{\al\in\fqr,\be\in\ffqs}\psi\lb\tqrp f(\al)+\tqsp f(\be)\rb,$$
$$T^{r,-s}_{f,\psi}=q^{-(r+s)/2}\sum_{\al\in\fqr,\be\in\ffqs}\psi\lb\tqrp f(\al)-\tqsp f(\be)\rb.$$
\end{lem}

\begin{proof} This follows immediately from Lemma \ref{lem0}.\end{proof}

\begin{lem}\label{rsd} Assume $r,s>0$, $r+s<d$. Let $\al\in\fqr,\be\in\ffqs$ be nonzero elements with monic minimal polynomials $g,h$ over $\fq$ respectively.  
For any natural $m$ denote $$A_m=\choice{\F_{q^{m/p}},}{p|m,}{\{0\},}{p\nmid m.}.$$
We have
\beq\label{rsd1}\afd{\psi(\tqrp f(\al)-\tqsp f(\be))}=\choice{1,}{g=h, p\deg g|r-s \mbox{ or } \al\in A_r,\be\in A_s, }{0,}{\mbox{otherwise}.}\eeq
\beq\label{rsd2}\afd{\psi(\tqrp f(\al)+\tqsp f(\be))}=\choice{1,}{g=h, p\deg g|r+s \mbox{ or } \al\in A_r,\be\in A_s, }{0,}{\mbox{otherwise}.}\eeq\end{lem}

\begin{proof} We prove \rf{rsd1}, \rf{rsd2} being similar. 
If $g=h$ then $\al,\be$ are conjugate over $\fq$ and so are $f(\al),f(\be)$, so $\tqkp f(\al)=\tqkp f(\be)$,
where $m=\deg g$. We have $$\tqrp f(\al)=\frac{r}{m}\tqkp f(\al),\tqsp f(\be)=\frac{s}{k}\tqkp$$ and so $$\tqrp f(\al)-\tqsp f(\be)=\frac{r-s}{m}\tqkp f(\al).$$
If $pm|r-s$ then $\frac{r-s}{m}$ is divisible by $p$ and so $\tqrp f(\al)-\tqsp f(\be)=0$ and $\psi(\tqrp f(\al)-\tqsp f(\be))=1$ for all $f\in\fd$ and
of course it also holds on average. If $pm\nmid r-s$ then denoting $l=(r-s)/m\bmod p$ we get from Lemma \ref{lem0} applied to the nontrivial character
$\psi^l$ instead of $\psi$ that the LHS of \ref{rsd1} equals 0.

For $\al\in A_r$ we have $\tqrp f(\al)=0$ because if $p|r$ we have $\tqrp f(\al)=p\cdot \t_{q^{r/p}/p}f(\al)=0$ and if $p\nmid r$
then $\al=0$ and $f(\al)=0$ for all $f\in\fd$. Similarly if $\be\in A_s$ then $\tqsp f(\be)=0$.
We see that if $\al\in A_r,\be\in A_s$ then $\psi(\tqrp f(\al)-\tqsp f(\be))=1$ for all $f\in\fd$ and the same holds for the average.

Now assume that $\be\in A_s$ but $\al\not\in A_r$. Then for all $f\in\fd$ we have $\psi(\tqrp f(\al)-\tqsp f(\be))=\psi(\tqrp f(\al))$ and so
$$\afd{\psi(\tqrp f(\al)-\tqsp f(\be))}=\afd{\psi(\tqrp f(\al))}=0$$ by Lemma \ref{lem0}, since $s<d$ and $\al\not\in A_r$.
The case $\al\in A_r,\be\not\in A_s$ is treated similarly.

Finally assume that $g\neq h$ and $\al\not\in A_r,\be\not\in A_s$, i.e. $\al,\be\neq 0$, $pu\nmid r,pv\nmid s$, where $u=\deg g,v=\deg h$ (we have $u|r,v|s$).
As in the proof of Lemma \ref{lem0} we may average over
$$\fd''=\{f\in\fq[x]|\deg f=d,f(0)=0\}$$ instead of $\fd$ (this does not change the average). Now since $u+v\le r+s<d$,
the map $\pi:\fd''\to\fq[x]/gh$ defined by $\pi(f)=(f/x)\bmod h$ is exactly $(q-1)q^{d-u-v-1}$ to one and since $x$ is invertible
modulo $gh$ (as $\al,\be\neq 0$) so is the map $\pi':\fd''\to\fq[x]/gh$ defined by $\pi'(f)=f\bmod gh$. However by the Chinese remainder
theorem we have $\fq[x]/gh\cong\fq[x]/g\times\fq[x]/g\cong\fqu\times\fqv$ (direct product of rings) with the isomorphism given by
$f\mapsto (f(\al),f(\be))$. We conclude that as $f$ ranges over $\fd''$ each pair $(f(\al),f(\be))\in\fqu\times\fqv$ is obtained
equally many times. However since $pu\nmid r$ the map $\sig:\fqu\to\fp$ defined by $\sig(\gam)=\psi(\tqrp\gam)=\psi(\t_{q^u/p}\gam)^{r/u}$
also assumes every value equally many times and the same goes for $\psi(\tqrp\gam)$ on $\fqv$. We conclude that as $f$ ranges over $\fd''$
each $p$-th root of unity occurs equally many times as $\psi(\tqrp f(\al)-\tqrp f(\be))$ and since averaging over $\fd$ is equivalent
to averaging over $\fd''$ we obtain \rf{rsd1}.

\end{proof}

Now denote by $\pi(m)$ the number of monic irreducible polynomials in $\fq[x]$ with degree $m$.

\begin{prop}\label{mdrs1}Assume $r,s>0$ and $r+s<d$. Then 
\begin{multline}\label{mdrs11} M_d^{r,-s}=q^{-(r+s)/2}\lb\sum_{m|(r,s)\atop{mp|r-s\atop{mp\nmid r}}}\pi(m)m^2+e_{p,r}e_{p,s}q^{(r+s)/p}\rb+\\
+q^{-(r+s)/2}\lb(1-\epr)\eeps q^{s/p}+(1-\eeps)\epr q^{r/p}+(1-\epr)(1-\eeps)\rb,\end{multline}
\begin{multline}\label{mdrs12} M_d^{r,s}= q^{-(r+s)/2}\lb\sum_{m|(r,s)\atop{mp|r+s\atop{mp\nmid r}}}\pi(m)m^2+e_{p,r}e_{p,s}q^{(r+s)/p}\rb+\\
+q^{-(r+s)/2}\lb(1-\epr)\eeps q^{s/p}+(1-\eeps)\epr q^{r/p}+(1-\epr)(1-\eeps)\rb.\end{multline}
\end{prop}

\begin{proof} We prove \rf{mdrs11}, \rf{mdrs12} begin similar. By Lemma \ref{trs} we have 
\beq\label{mrs}M^{r,-s}_d=q^{-(r+s)/2}\sum_{\al\in\fqr\atop{\be\in\ffqs}}\afd{\psi(\tqrp f(\al)-\tqsp f(\be))}.\eeq
Now we can use Lemma \ref{rsd} to evaluate this expression. Denote by $g,h$ the monic minimal polynomials (over $\fq$) of $\al\in\fqr,\be
\in\ffqs$ respectively. First we count the contribution of those $\al,\be$ for which $g=h$ (i.e. they are conjugate), $p\deg g|r-s$ and $p\deg g\nmid r$
(and consequently $p\deg g\nmid s$),
so $\al\not\in A_r$ and $\be\not\in A_s$. 
Each polynomial $g$ has exactly $\deg g$ roots and the number of pairs $\al,\be$ with minimal polynomial $g$ is $(\deg g)^2$.
The contribution of all such $g$ to the sum in \rf{mrs} is $$\sum_{m|(r,s)\atop{mp|r-s\atop{mp\nmid r}}}\pi(m)m^2.$$

It remains to evaluate the contribution of the pairs $\al\in A_r,\be\in A_s$ (see the notation in the previous lemma).
If $p|(r,s)$ then every pair $\al\in\fqrp,\be\in\fqsp$ contributes 1 to the sum (by Lemma \ref{rsd}) and the number of such pairs is
$q^{(r+s)/p}$. If $p|r$ but $p\nmid s$ then the pairs $\al\in\fqrp,\be=0$ (and only them) contribute 1, so we get a total contribution
of $q^{r/p}$. The case $p\nmid r,p|s$ is treated similarly. If $p\nmid rs$ then only $\al=\be=0$ adds 1 to the sum.
In any case we get the value stated in the proposition.

To prove \rf{mdrs12} we proceed similarly using \rf{rsd2}.
\end{proof}

\begin{thm}\label{thm1p} Assume $r\ge s>0$ and $r+s<d$. Then 
$$M^{r,-s}_d=\del_{r,s}r+O\lb rq^{-r/2}+q^{(1/p-1/2)(r+s)}\rb,$$
where $$\del_{r,s}=\choice{1,}{r=s,}{0,}{r\neq s.}$$
We also have $$M^{r,s}_d=O\lb rq^{(1/p-1/2)(r+s)}\rb.$$\end{thm}

\begin{proof} By the Proposition \ref{mdrs1} we have $$M_d^{r,-s}=\sum_{m|(r,s)\atop{mp|r-s\atop{mp\nmid r}}}\pi(m)m^2+O\lb q^{(1/p-1/2)(r+s)}\rb.$$
It is well known that $\pi(m)=q^m/m+O(q^{m/2}/m)$ (see \cite[\S 2]{rosen}). If $r=s$ then
$$\sum_{m|r\atop{mp\nmid r}}\pi(m)m^2=rq^r+O(rq^{r/2})$$ (note that except for $m=2$ and possibly $m=r/2$ the other terms are negligible),
which implies the assertion of the theorem for the case $r=s$.
If $r>s$ then $(r,s)\le s/2$ and $$\sum_{m|(r,s)\atop{mp|r-s\atop{mp\nmid r}}}\pi(m)m^2=O(rq^{s/2}),$$ which implies the assertion of the theorem for $r\neq s$.
The second part of the Theorem follows similarly from the second part of Proposition \ref{mdrs1} 
(note that this time only $m$ s.t. $mp|r+s$ contribute to the sum).
\end{proof}

The last result should be compared with the following (see \cite[Thm. 2]{ds}):
$$\la T_U^r\overline{T_U^s}\ra_{U\in\U_N}=\del_{r,s}\min(r,N),$$ $$\la T_U^rT_U^s\ra_{U\in\U_N}=0$$
(for $r,s>0$).

\subsection{Proof of Theorem \ref{thm2p}}\label{pthm2p}

For simplicity we will prove Theorem \ref{thm2p} for $\th=0$, the proof of the general case proceeds with only slight modifications.
We will write $S^2(f,\psi)$ instead of $S^2_0(f,\psi)$ for the 2-level density function.

\begin{lem}$$S^2(f,\psi)=\frac{1}{(d-1)^2}\sum_{r,s=-\ity}^\ity\hw\lb\frac{r}{d-1},\frac{s}{d-1}\rb\lb T^{r,s}_{f,\psi}-T^{r+s}_{f,\psi}\rb.$$
\end{lem}

\begin{proof} We have \beq\label{breaks}S^2(f,\psi)=\sum_{j,k=1}^{d-1}v_{d-1}(\th_j,\th_k)-\sum_{j=1}^{d-1}v_{d-1}(\th_j,\th_j).\eeq
The Fourier series coefficients of the bi-periodic function $v_{d-1}$ are given by
\begin{multline*}\hat{v}_{d-1}(r,s)=\frac{1}{4\pi^2}\int_0^{2\pi}\int_0^{2\pi}v_{d-1}(t,u)e^{-irt-isu}\dd t\dd u=\\=\frac{1}{(d-1)^2}\hv\lb\frac{r}{d-1},\frac{s}{d-1}\rb,\end{multline*}
the derivation is by a standard calculation, similar to that of \ref{fourlem}.
Since $v_{d-1}$ is smooth the following holds for $|z|=|w|=1$:
\begin{multline*}v_{d-1}(\arg z,\arg w)=\sum_{r,s=-\ity}^\ity\hat{v}_{d-1}(r,s)z^rw^s=\\=\frac{1}{(d-1)^2}\sum_{r,s=-\ity}^\ity\hw\lb\frac{r}{d-1},\frac{s}{d-1}\rb z^rw^s.\end{multline*}
Now if $\rho_1,...,\rho_{d-1}$ are the normalised zeroes of $L_{f,\psi}$ and $\rho_j=e^{i\th_j}$ then
\begin{multline}\label{mu1}\sum_{j,k=1}^{d-1}v_{d-1}(\th_j,\th_k)=
\frac{1}{(d-1)^2}\sum_{j,k=1}^{d-1}\sum_{r,s=-\ity}^\ity\hw\lb\frac{r}{d-1},\frac{s}{d-1}\rb\rho_j^r\rho_k^s=\\
=\frac{1}{(d-1)^2}\sum_{r,s=-\ity}^\ity\hw\lb\frac{r}{d-1},\frac{s}{d-1}\rb T^{r,s}_{f,\psi},\end{multline}
and 
\begin{multline}\label{mu2}\sum_{j=1}^{d-1}v_{d-1}(\th_j,\th_j)=\frac{1}{(d-1)^2}\sum_{j=1}^{d-1}\sum_{r,s=-\ity}^\ity\hw\lb\frac{r}{d-1},\frac{s}{d-1}\rb\rho_j^{r+s}=\\
=\frac{1}{(d-1)^2}\sum_{r,s=-\ity}^\ity\hw\lb\frac{r}{d-1},\frac{s}{d-1}\rb T^{r+s}_{f,\psi}.\end{multline}
Combining \rf{breaks},\rf{mu1} and \rf{mu2} we obtain the statement of the Lemma.
\end{proof}

\begin{cor}\label{corp}$$\afd{S^2(f,\psi)}=\frac{1}{(d-1)^2}\sum_{r,s=1}^\infty\hw\lb\frac{r}{d-1},\frac{s}{d-1}\rb
\lb M^{r,s}_d-M^{r+s}_d\rb.$$
\end{cor}

\begin{proof} Just average the previous lemma over $\fd$.\end{proof}

We need one more lemma:

\begin{lem}\label{flem} Let $V\in\srr$ be a window function, $\hv$ its Fourier transform. Denote $K(\sig)=\max(1-|\sig|,0)$.
We have \begin{multline*}\inti\hv(\sig,-\sig)K(\sig)\dd \sig=
\hv(0,0)-\\-\frac{1}{4\pi^2}\inti\inti V(t,u) \lb 1-\lb\frac{\sin((t-u)/2)}{(t-u)/2}\rb^2\rb\dd t\dd u.\end{multline*}\end{lem}

\begin{proof} Define $Y(\tau)=\inti V(t+\tau,t)\dd t$. We have $Y\in\sr$. It is easy to see from the definitions and Fubini's theorem that the Fourier
transform of $Y(\tau)$ is $\hat{Y}(\sig)=2\pi\hv(\sig,-\sig)$. The Fourier transform of the function $\lb\sin (\tau/2)/(\tau/2)\rb^2$
is $K(\sig)$, so we have by Plancherel's theorem
\begin{multline*}\inti\hv(\sig,-\sig)K(\sig)\dd\sig=\frac{1}{2\pi}\inti\hat{Y}(\sig)K(\sig)\dd\sig
\\=\frac{1}{4\pi^2}\inti Y(\tau)\lb\frac{\sin(\tau/2)}{\tau/2}\rb^2\dd\tau=\\
=\frac{1}{4\pi^2}\inti\inti V(t,u)\lb \frac{\sin((t-u)/2)}{(t-u)/2}\rb^2\dd t\dd u=\\
=\hv(0,0)-\frac{1}{4\pi^2}\inti\inti V(t,u) \lb 1-\lb\frac{\sin((t-u)/2)}{(t-u)/2}\rb^2\rb\dd t\dd u.\end{multline*}
\end{proof}

Now we are ready to prove Theorem \ref{thm2p}. Assume $p>2$. Let $V\in\srr$ be s.t. $\hw$ is supported on $|\eta|+|\xi|\le 1$. By Corollary \ref{corp} we have
\beq\label{msum}\afd{S^2(f,\psi)}=\frac{1}{(d-1)^2}\sum_{r,s\in\Z\atop{|r|+|s|<d}}\hw\lb\frac{r}{d-1},\frac{s}{d-1}\rb \lb M^{r,s}_d-M^{r+s}_d\rb.\eeq

First we bound the contribution to the sum \rf{msum} of $r,s$ s.t. $r\neq -s$. By Theorem \ref{thm1} we have
$$\sum_{r,s\in\Z\atop{|r|+|s|<d\atop{r\neq -s}}}M_d^{r+s}=O\lb d\sum_{r=1}^d M_d^r\rb=O(d)$$
(note that $M_d^r=M_d^{-r}$), since $M_d^r$ decreases geometrically in $r$ for $0<r<d$.
Similarly, by Theorem \ref{thm1p} we also have
$$\sum_{r,s\in\Z\atop{|r|+|s|<d\atop{r\neq -s}}}M_d^{r,s}=O(d)$$ as $M_d^{r,s}$ for $r\neq s$ decreases geometrically in $|r|+|s|$.
The overall contribution to the RHS of \rf{msum} is $O(1/d)$ (note that $\hv$ is bounded).

It remains to estimate \begin{multline*}\frac{1}{(d-1)^2}\sum_{-d\le r\le d}\hv\lb\frac{r}{d-1},\frac{s}{d-1}\rb \lb M_d^{r,-r}-M_d^0\rb=\\=
\frac{d}{d-1}\hv(0,0)+\frac{1}{d-1}\sum_{-d/2<r<d/2\atop{r\neq 0}}\hv\lb\frac{r}{d-1},\frac{-r}{d-1}\rb\lb M_d^{r,-r}-d+1\rb
\end{multline*} (note that $M_d^0=d-1,M_d^{0,0}=(d-1)^2$).
Invoking Theorem \ref{thm1p} and noting that the error terms accumulate to at most $O(1/d^2)$ 
(the error term for $M_d^{r,-r}$ in Theorem \ref{thm1p} decreases
geometrically in $r$) we see that
\begin{multline*}\afd{S^2(f,\psi)}=\\=\hv(0,0)+\frac{1}{(d-1)^2}\sum_{-d/2< r< d/2\atop{r\neq 0}} \hw\lb\frac{r}{d-1},\frac{-r}{d-1}\rb(|r|-d+1)+O(1/d)=\\
=\hv(0,0)+\sum_{-d/2<r<d/2\atop{r\neq 0}}\lb\frac{|r|}{d-1}-1\rb\hw\lb\frac{r}{d-1},\frac{-r}{d-1}\rb\frac{1}{d-1}+O(1/d)\to\\\to \hv(0,0)+\inti \hv(\sig,-\sig)(|\sig|-1)\dd\sig=
\hv(0,0)-\inti \hv(\sig,-\sig)K(\sig)\dd\sig\end{multline*}
as $d\to\ity$ by the definition of the Riemann integral (we used the fact that $\hv(\sig,-\sig)$ is supported on $\lbb -1/2,1/2\rbb$).
Now using Lemma \ref{flem} we obtain the assertion of Theorem \ref{thm2p}.


\section{Reformulation in terms of Dirichlet \\L-functions and generalisation}\label{secdirg}

In the present section we will see that the family of L-functions $L_{f,\psi},f\in\fd$ is actually a special case of a family of Dirichlet L-functions corresponding to multiplicative
characters of $\F_q[x]$ and generalise our main results to such families.

\subsection{Dirichlet characters and L-functions}\label{secdir}

We briefly recall the basic properties of Dirichlet characters and L-functions over $\fq[x]$. For details see \cite[\S 4]{rosen}.
Let $Q(x)\in\fq[x]$ be a monic polynomial of degree $m$ and let $\chi:\lb\fq[x]/Q\rb^\times\to\C^\times$ be a character of the multiplicative
group of residues modulo $Q$. We may extend $\chi$ to $\fq[x]$ by $$\chi(g)=\choice{\chi(g\bmod Q),}{(g,Q)=1,}{0,}{\mbox{otherwise}.}$$
The map $\chi:\fq[x]\to\C$ thus defined is called a Dirichlet character.
The character $\chi$ is called primitive if there is no proper divisor $Q_1$ of $Q$ s.t. $\chi(g)$ for $g$ prime to $Q$ only depends on $g\bmod Q_1$.
It is called trivial if it takes the value 1 on all polynomials prime to $Q$. 
It is called even if it takes the value 1 on constants and odd otherwise. We denote $$e(\chi)=\choice{1,}{\chi\mbox{ is even},}{0,}{\chi\mbox{ is odd}.}$$

Denote by $\M$ the set of monic polynomials in $\fq[x]$ and by $\p$ the set of monic irreducible polynomials in $\fq[x]$. 
The L-function corresponding to the character $\chi$ is defined as follows:
\beq\label{dirl}L_\chi(z)=\sum_{g\in\M}\chi(g)z^{\deg g}=\prod_{h\in\p}(1-\chi(h)z)^{-1}.\eeq
It turns out that for a primitive character $\chi$ modulo $Q$ the function $L_\chi(z)$ is a polynomial of degree $d-2$ if $\chi$ is even and
$d-1$ if $\chi$ is odd. Further it factors as follows:
$$L_\chi(z)=(1-z)^{e(\chi)}\prod_{i=1}^{d-1-e(\chi)}(1-\rho_iq^{1/2}z)$$ with $|\rho_i|=1$. The $\rho_i$ are called the normalised zeroes of $L_\chi(z)$.

\subsection{Dirichlet characters corresponding to A-S curves}\label{diras}

Let $d$ be a natural number and $\psi\in\Psi$ a nontrivial additive character of $\F_p$. Let $f=\sum_{i=0}^da_ix^i\in\fd$ be a polynomial. 
We will attach
a Dirichlet character $\chi$ modulo $x^{d+1}$ to $f,\psi$. Let $g=\sum_{i=1}^kb_ix^i\in\fq[x]$ be a polynomial. If $x|g$ we define $\chi(g)=0$.
Otherwise we may write \beq\label{gpr}g(x)=\prod_{i=1}^k(1-\al_ix),\eeq where $\al_i$ are the roots of $g$ in the algebraic closure of $\fq$. 

First we observe that for every natural $j$ the quantity $\sum_{i=1}^k\al_i^j$ lies in $\fq$ and depends only on the coefficients
$b_0,...,b_j$ (in other words it depends only on $g\bmod x^{d+1}$).
As in section \ref{pthm1} we denote by $\sig_j$ the order $j$ elementary symmetric function in the variables $x_1,...,x_k$.
By \rf{gpr} we have $\sig_j(\al_1,...,\al_k)=(-1)^jb_j/b_0$. Using Newton's identity \ref{newton} recursively we can show that
$\sum_{i=1}^kx_i^j$ can be expressed as a polynomial in $\sig_l,1\le l\le \min(j,k)$ with integer coefficients. Therefore $\sum_{i=1}^k\al_i^j$
can be expressed as a polynomial in $b_1/b_0,...,b_j/b_0$ with integer coefficients and so it must lie in $\fq$ and depends only on\\
$b_1/b_0,...,b_j/b_0$ (and therefore only on $g\bmod x^{d+1}$).

It follows from the above that the quantity $$\sum_{i=1}^k f(\al_i)=\sum_{j=0}^d a_j\sum_{i=1}^k\al_i^j$$ is in $\fq$ and depends only on
$b_1/b_0,...,b_{\min(d,k)}/b_0$. Now we define \beq\label{dir}\chi(g)=\psi\lb\tqp\sum_{i=1}^k f(\al_i)\rb.\eeq
By what we have seen $\chi(g)$ is well defined and depends only on $g\bmod x^{d+1}$. Further it is multiplicative, because the set of zeroes
(counting multiplicity) of a product of two polynomials is the union of their sets of zeroes. Therefore $\chi$ is a Dirichlet character
modulo $x^{d+1}$. Obviously for constant $g$ we have $\chi(g)=1$, so $\chi$ is even. By \ref{dir} we also have that $\chi^p$ is trivial,
so $\chi$ is an order $p$ character modulo $x^{d+1}$.

We fix $\psi\in\Psi$ and denote by $\chi_f$ the character constructed above for a given $f\in\fd$.

\begin{lem} For any $f\in\fd$ the character $\chi_f$ is primitive. The characters $\chi_f,f\in\fd$ are all distinct and any primitive
character $\chi$ modulo $x^{d+1}$ with $\chi^p$ trivial is of the form $\chi=\chi_f$ for some $f\in\fd$.\end{lem}

\begin{proof} First we show that for $f\in\fd$ the character $\chi=\chi_f$ is primitive. For this it is enough to show that for some $c\in\fq$
we have $\chi(1-cx^d)\neq 1$. Write $f=\sum_{i=0}^da_ix^i$. It is easy to see from the definition that $\chi(1-cx^d)=\psi(\tqp (cda_d))$.
Since $(d,p)=1$ and $a_d\neq 0$ there exists $c\in\fq$ s.t. $\tqp (cda_d)\neq 0$.

Now let $f_1,f_2\in\fd$ s.t. $f_1\neq f_2$. We will show that $\chi_{f_1}\neq\chi_{f_2}$.
By the definition of $\fd$ there exists some $e\le d$ s.t. 
$f=f_1-f_2\in\fe$. Let $g\in\fq[x]$ be some polynomial prime to $x$. Write it as $g(x)=g(0)\prod_{i=1}^k(1-\al_ix)$ with 
$\al_i\in\fqr$ for some $r$. We have $$\chi_{f_1}(g)\chi_{f_2}(g)^{-1}=\psi\lb\tqp\sum_{i=1}^k f(\al)\rb.$$
It is enough to show that for some $g$ we have $\tqp\sum_{i=1}^k f(\al)\neq 0$.
Taking $g=1-cx^e$ we have (as above) $\chi(g)=\psi(\tqp(cea_e))$ where 
$f=\sum_{i=0}^ea_ix^i$ and using the fact that $(e,p)=1$ and $a_e\neq 0$ we can pick $c$ so that $\chi(g)\neq 1$.

We have seen that the correspondence $f\mapsto\chi_f$ is one-to-one from $\fd$ to the set of order $p$ primitive characters modulo $x^{d+1}$.
To show that it is onto it is enough to show that these sets are identical in size. Recall that $\#\fd=(q-1)q^{d-\floor{d/p}-1}$.
For a finite abelian group $G$ denote by $G^*$ its dual and by $G[m]$ the $m$-torsion of the group $G$. 
The groups $G,G^*$ are always isomorphic. Take $G=(\fq[x]/x^{d+1})^\times$. First we compute $\#G^*[p]=\#G[p]$, which is the number
of all order $p$ characters modulo $x^{d+1}$. Each element of $G$ is represented uniquely by a polynomial $g(x)=\sum_{i=0}^dc_ix^i$.
We have $g(x)^p=\sum_{i=1}^dc_ix^{pi}$ and $g(x)^p\equiv 1\pmod{x^{d+1}}$ iff $c_1=c_2=...=c_\floor{d/p}=0$ and $c_0=1$.
Therefore $\#G^*[p]=\#G[p]=q^{d-\floor{d/p}+1}$ and this is also the number of order $p$ characters modulo $x^{d+1}$.
By the same reasoning the number of order $p$ characters modulo $x^d$ is $q^{d-\floor{(d-1)/p}}=q^{d-\floor{d/p}}$ (since $(d,p)=1$)
and so 
the number of primitive characters modulo $x^{d+1}$ equals $q^{d-\floor{d/p}+1}-q^{d-\floor{d/p}}=(q-1)q^{d-\floor{d/p}-1}=\#\fd$.
\end{proof}

\begin{lem}\label{lemm} For the character $\chi$ defined above we have $L_\chi(z)=(1-z)L_{f,\psi}$. In particular the normalised nontrivial zeroes
of $L_{\chi}(z)$ coincide with the normalised zeroes of $L_{f,\psi}$.\end{lem}

\begin{proof} Since both $L_\chi(z)$ and $(1-z)L_{f,\psi}$ have constant coefficient 1 it it enough to show that
$$\frac{\dd}{\dd z}\log L_\chi(z)=\dz\log\lb(1-z)L_{f,\psi}\rb.$$ 
Let $r$ be a natural number, $\al\in\fqr$ an element with minimal polynomial $h$ (over $\fq$) of degree $s|r$.
Denote by $\al_1=\al,...,\al_s$ the roots of $h$ in $\fqr$. We have $\tqrp(f(\al))=\sum_{i=1}^r f(\al^{p^i})=\frac{r}{s}\sum_{i=1}^s f(\al_i)$
and so for $\al\neq 0$ we have $\psi(\tqrp(f(\al)))=\chi(h^*)^{r/s}$ (recall that $h^*=\sum_{i=0}^sc_{s-i}x^i$ where
$h=\sum_{i=0}^sx^i$ and it has roots $\al_1^{-1},...,\al_s^{-1}$). For $\al=0$ we have $\tqrp(f(\al))=rf(0)$.

Using the above and \ref{dirl} we obtain \begin{multline*}\dz\log L_{f,\psi}=\sum_{r=1}^\infty \sum_{\al\in\fqr}\psi(\tqrp f(\al))z^{r-1}=\\
=\frac{1}{z}\sum_{r=1}^\ity\sum_{h\in\p,\deg h|r,h\neq x}\deg(h)\chi(h^*)^{r/\deg(h)}z^r+\frac{1}{z}\sum_{r=1}^\ity\psi(\tqp f(0))^rz^r=\\=
\frac{1}{z}\sum_{h\in\p}\sum_{k=1}^\ity\deg(h)\chi(h)^kz^{k\deg h}+\frac{1}{1-z}
=\sum_{h\in\p}\frac{\deg(h)\chi(h)z^{\deg h-1}}{1-\chi(h)z^{\deg h}}+\frac{1}{1-z}=\\=
\dz\log\lb\prod_{h\in\p}(1-\chi(h)z^{\deg h})^{-1}\rb-\dz\log(1-z)=\dz\log\lb L_{\chi_f}(1-z)^{-1}\rb \end{multline*}
(we used the fact that the operation $h\mapsto h^*$ permutes $\p\sm\{x\}$, that $\chi(x)=0$ and that $f(0)=0$ for $f\in\fd$).
\end{proof}

We see that the family of L-functions $L_{f,\psi}(z),f\in\fd$ coincides with the family of L-functions $L^*_\chi(z)=(1-z)^{-1}L_\chi(z)$ where $\chi$
ranges over the order $p$ primitive characters modulo $x^{d+1}$. Next we generalise Theorems \ref{thm1} and \ref{thm2} to more general
families of Dirichlet characters, obtaining a new (but essentially equivalent) proof of our results for A-S L-functions.

\subsection{The family of Dirichlet L-functions corresponding to a subgroup of $(\fq[x]/Q)^{\times}$}

For any finite Abelian group $A$ we denote by $A^*$ its dual group. Let $Q(x)\in\fq[x]$ be a monic polynomial of degree $m$. 
Denote by $G$ the group of characters modulo $Q$, which we will also identify with $(\fq[x]/Q)^{\times*}$,
i.e. we view the elements of $G$ also as characters of $(\fq[x]/Q)^\times$.
Let $H$ be a subgroup of $G$. For any $Q'|Q$ we denote by $H_{Q'}$ the subgroup of $H$ consisting of the characters which have period $Q'$ (or a divisor
of $Q'$). Denote by $H'$ the set of primitive characters in $H$. Denote by $H^\bot$ the subgroup of $(\fq[x]/Q)^\times$ consisting of elements
$g$ s.t. $\chi(g)=1$ for all $\chi\in H$. It is the subgroup of $(\fq[x]/Q)^\times$ orthogonal to $H\ss(\fq[x]/Q)^{\times*}$ and its order is
$\#H^\bot=\#G/\#H$ (this relation holds for any finite abelian group).
The following orthogonality relation holds for $g\in\fq[x]$:
\beq\label{ort} \sum_{\chi\in H}\chi(g)=\choice{\#G/\#H,}{g\bmod Q\in H^\bot}{0,}{\mbox{otherwise}.}\eeq
We denote by $\p(H)$ the set of monic irreducible polynomials $h$ s.t. $h\bmod Q\in H^\bot$.

Let $\chi$ be an even primitive character modulo $Q$. Recall that its L-function can be factored as \beq\label{dirlzeros}L_\chi(z)=(1-z)\prod_{i=1}^{m-2}(1-q^{1/2}\rho_i),\eeq with $|\rho_i|=1$ ($\rho_i$ are the normalised zeroes of the L-function).

\begin{lem}\label{1}Let $r$ be a natural number. 
$$\sum_{i=1}^{m-2}\rho_i^r=-q^{-r/2}-q^{-r/2}\sum_{h\in\p,\deg h|r}(\deg h)\chi(h)^{r/\deg h}.$$\end{lem}

\begin{proof} By \ref{dirl} and $\ref{dirlzeros}$ we have
\begin{multline*}\sum_{r=1}^\ity q^{r/2}\lb q^{-r/2}+\sum_{i=1}^{m-2}\rho_i^r\rb z^r=z\lb \frac{1}{1-z}+\sum_{i=1}^{m-2}\frac{q^{1/2}\rho_i}{1-q^{1/2}\rho_iz}\rb=\\
=-z\dz\log\lb (1-z)\prod_{i=1}^{m-2} (1-q^{1/2}\rho_iz)\rb=-z\dz\log L_\chi(z)=\\=-\sum_{h\in\p}\frac{\chi(h)(\deg h)z^{\deg h}}{1-\chi(h)z^{\deg h}}
=\sum_{r=1}^\ity\sum_{h\in\p}(\deg h)\chi(h)^rz^{r\deg h}=\\=\sum_{r=1}^\ity\lb\sum_{h\in\p,\deg h|r}(\deg h)\chi(h)^{r/\deg h}\rb z^r.\end{multline*}
Comparing coefficients at $z^r$ we obtain the statement of the lemma.\end{proof}

We denote $T_\chi^r=\sum_{i=1}^{m-2}\rho_i$. For a group of Dirichlet characters $J$ modulo $Q$ and a natural number $s$ denote
$\eta(J,s)=\#\{h\in\p(J)|\deg h=s$.
For a nonzero polynomial $P\in\fq[x]$ with factorisation $P=P_1...P_k$ into irreducibles we denote
$$\mu(P)=\choice{(-1)^k,}{P\mbox{ is squarefree},}{0,}{\mbox{otherwise}}$$
(this is the M\"{o}bius function on $\fq[x]$). For an abelian group $A$ and natural number $k$ denote by $A^k$ the subgroup
of $k$-th powers in $A$. Finally denote by $H^\pr$ the set of primitive characters in $H$
and $M_H^r=\la T_\chi^r\ra_{\chi\in H^\pr}$.
The following proposition is a generalisation of Proposition \ref{irr} to an arbitrary family
of Dirichlet characters corresponding to a subgroup $H$ of characters modulo $Q$.

\begin{prop}\label{dirprop}
$$M_H^r=\la T_\chi^r\ra_{\chi\in H^\pr}=-q^{r/2}-\frac{q^{-r/2}}{\#H^\pr}\sum_{Q'|Q}\mu\lb\frac{Q}{Q'}\rb\# H_{Q'}\sum_{s|r}s\cdot \eta\lb H_{Q'}^{r/s},s\rb,$$
where $\sum_{Q'|Q}$ denotes summation over monic divisors $Q'$ of $Q$.\end{prop}

\begin{proof}
First of all it follows from the inclusion-exclusion principle that for any map $X:H\to\C$ we have
\beq\label{ie}\sum_{\chi\in H^\pr}X(\chi)=\sum_{Q'|Q}\mu(Q/Q')\sum_{\chi\in H_{Q'}}X(\chi).\eeq
Next, by Lemma \ref{1} and the orthogonality relation \rf{ort} for any $Q'|Q$ we have
\begin{multline*}\sum_{\chi\in H_{Q'}}(T_\chi^r+q^{r/2})=-q^{-r/2}\cdot\#H_{Q'}\cdot\sum_{h\in\p\atop{\deg h|r\atop{h^{r/\deg h}\bmod Q'\in \lb H_{Q'}^{r/s}\rb^\bot}}}\deg h=
\\=-q^{-r/2}\cdot\#H_{Q'}\sum_{s|r}s \cdot \eta\lb H_{Q'}^{r/s},s\rb\end{multline*}
(we used the fact that $h^{r/s}\bmod Q'\in H_{Q'}^\bot$ iff $h\bmod Q'\in\lb H_{Q'}^{r/s}\rb^\bot$).
Combining this with \rf{ie} we obtain the statement of the proposition.
\end{proof}

The last proposition can be used to obtain bounds on $M_H^r$. For example assume that $Q$ is irreducible. Then it follows from the proposition
that for $r\ge m$ we have $M_H^r=O\lb rq^{r/2-m}\#H^\bot\rb$, since the total number of monic polynomials $h$ with $\deg h=r,h\bmod Q\in H^\bot$ is
$q^{r-m}\#H^\bot$. For $r<m$ we have $M_H^r=O\lb rq^{-r/2}\#H^\bot\rb$.

For another example take $Q=x^{d+1}$ and $H=\lb(\fq[x]/Q)^\times\rb^*[p]$ (the group of order $p$ characters modulo $Q$). We have seen in section
\ref{diras} that $H^\pr=\{\chi_f\}_{f\in\fd}$, so $M_H^r=M_d^r$. Proposition \ref{irr} now follows from proposition \ref{dirprop}. Indeed
the only divisors $Q'|Q$ s.t. $\mu(Q/Q')\neq 0$ are $Q'=x^{d+1},x^d$ and
$H_{Q'}^\bot$ consists of the (invertible) $p$-th powers modulo $Q'$, which are represented by polynomials of the form
$a_0+a_1x^p+...+a_\floor{d'/p}x^{\floor{d'/p}p}$ where $d'=\deg Q'-1$. Combining this with the fact that $H^k=H$ if $(k,p)=1$ and $H^k$ is trivial
if $p|k$ the formula in proposition \ref{dirprop} translates into the one in proposition \ref{irr} (note that for $s<d$ we have
$\et{H^\bot,s}=0$, since the classes in $H^\bot$ are represented by $p$-th powers). 
Our main theorems follow from the latter proposition.

\section{A-S family with odd polynomials}\label{odd}

Throughout this section $p>2$. Let $d$ be odd and as usual $(d,p)=1$. Denote by $\od$ the subset of odd polynomials in $\fd$, i.e.
polynomials $f\in\fd$ satisfying $f(-x)=-f(x)$, in other words only odd powers of $x$ appear in $f$.
We call $\od$ as well as the corresponding family of curves and L-functions the odd A-S family. As with $\fd$ the family
$\{L_{f,\psi}\}_{f\in\od}$ does not depend on the choice of $\psi\in\Psi$.
In the present section we formulate conjectures for $\od$ analogous to our main results for $\fd$ based on
a random symplectic matrix model. However we will only be able to prove a very weak result in this direction.

\begin{lem} For $f\in\od,\psi\in\Psi$ we have $L_f(z)\in\R[z]$.\end{lem}

\begin{proof} By \rf{deflpsi} is would suffice to show that
$\sum_{\al\in\fqr}\psi(\tqrp f(\al))\in\R$ for every natural $r$.
Since $f$ is an odd polynomial we have $f(-\al)=-f(\al)$ for $\al\in\fqr$ and so
$\psi(\tqrp f(-\al))=\ol{\psi(\tqrp f(\al))}$ and partitioning $\fqr\sm\{0\}$ into pairs $\al,-\al$ we obtain $\sum_{\al\in\fqr}\psi(\tqrp f(\al))\in \R$.\end{proof}

The latter fact suggests that we model the set of normalised L-zeroes of a random $f\in\od$ by the set of eigenvalues of a random matrix $U\in\usp_{d-1}$
(we denote thus the unitary symplectic group), taken uniformly w.r.t. the Haar measure. Note that the characteristic polynomial of a unitary
symplectic matrix has real coefficients. This is the model usually used for the L-zeroes of a family
of curves over a finite field, provided that the corresponding L-functions do not split into primitive L-functions with non-real coefficients,
as it happens for the entire A-S family if $p>2$. 
A more compelling reason for considering the random symplectic matrix model is an equidistribution result due to N. Katz and P. Sarnak for similar
(and more general) families of L-functions with fixed $d$ and $q\to\infty$, see Theorem 3.10.7 in \cite{monper}. 

For the unitary symplectic group the following holds:
$$\la T_U^r\ra_{U\in\usp_{d-1}}=\choice{-e_{2,r},}{r<d,}{0,}{r\ge d,}$$
see \cite[\S 4]{ds}. 

We conjecture the following

\begin{conj}\label{thm1o} There exists a constant $\del>0$ such that for any $\eps>0$ we have 
$$\aod{T_{f,\psi}^r}=\left\{\begin{array}{ll}{-e_{2,r},} & {r<d} \\ {0,} & {r\ge d}\end{array}\right\}+O_\eps\lb q^{\eps r-\del d}+q^{-\del r}\rb.$$\end{conj}

From this one can derive using the method of section \ref{pthm2} the following

\begin{conj} Let $V\in\sr$ be a window function, $$v_d(t)=\sum_{n=-\ity}^\ity V(d(t+2\pi n)).$$
For $f\in\od$ denote $Z_f=\sum_{j=1}^{d-1} v_d(\th_j)$, where $\rho_j=e^{i\th_j}$ are the
normalised zeroes of $L_{f,\psi}$. Similarly for a matrix $U\in\usp_{d-1}$ with eigenvalues $\rho_i$ denote $Z_U=\sum_{j=1}^{d-1} w_d(\th_j)$.
Then $$\aod{T_{f,\psi}^r}\to\la T_U^r\ra_{U\in\usp_{d-1}}$$ as $d\to\ity$.\end{conj}

A possible approach to estimating $\aod{T_{f,\psi}^r}$ is to reformulate the problem in terms of a family of Dirichlet characters and use Proposition \ref{dirprop}. 
We keep the notation of section \ref{secdir}.
We take $Q=x^{d+1}$. Recall that to any polynomial $f\in\fq[x]$ with $\deg f\le d$ we can attach a character $\chi_f$ modulo $x^{d+1}$
defined by \rf{gpr}, \rf{dir}. It is primitive iff $\deg f=d$. We also have $\chi_{f_1f_2}=\chi_{f_1}\chi_{f_2}$ if $\deg f_1,\deg f_2\le d$.
Denote $G=\lb(\fq[x]/Q)^\times\rb^*$. As usual we identify $G$ with the group of characters modulo $Q$. 
Denote $H=\{\chi_f|f\in\fq[x],\deg f\le r,f(-x)=-f(x)\}$. By the above remarks this is a subgroup of $G$. We have
$H^\pr=\{\chi_f\}_{f\in\od}$. By Lemma \ref{lemm} we have
$$\aod{T_{f,\psi}^r}=\la T_\chi^r\ra_{\chi\in H^\pr}.$$

The only monic divisor $Q'|Q$ s.t. $\mu(Q/Q')\neq 0$ is $Q'=x^d$ which we denote by $Q_1$.
Obviously for any natural $s|r$ we have $\eta(H^{r/s},s)=O(q^s)$. Also we have $\#H/\#H^\pr=q/(q-1),\#H_{Q_1}/\#H^\pr=1/(q-1)$.
Note also that $H^2=H$ since $H$ consists of order $p$ characters.
Proposition \ref{dirprop} now implies
\begin{prop}\label{oddprop}
\begin{multline*}\aod{T_{f,\psi}^r}=\\=-\frac{rq^{-r/2}}{q-1}\lb q\et{H,r}+e_{2,r}\frac{q}{2}\et{H,r/2}- \et{H_{Q_1},r}-e_{2,r}\frac{1}{2}\et{H_{Q_1},r/2}\rb+\\+O\lb rq^{-r/6}\rb.\end{multline*}\end{prop}

\begin{lem}\label{hort}The group $H^\bot$ consists of residues of the form $$g_1(x^p)g_2(x^2)\bmod Q$$ where $g_i\in\fqx,(g_i,x)=1$. The same holds for
$H_{Q_1}^\bot$ and $Q_1$ respectively.\end{lem}
\begin{proof} Take any $f\in\od$ and $g_1\in\fqx,(g_1,x)=1$. Since $\chi_f$ is an order $p$ character we have $\chi_f(g_2(x^p))=\chi_f(g_1(x)^p)=1$.
Now take $g_2\in\fqx,(g_2,x)=1$. We may assume $g_2=1+b_1x+...+b_kx^k$ since $\chi_f$ is even. The inverse roots of $g_2(x)$ come in pairs $\pm\al$
and since $f$ is odd by \ref{dir} we have $\chi_f(g_2(x^2))=1$. It remains to note that the group of residues of the form $g_1(x^p)g_2(x^2)\bmod Q$ 
has order $(q-1)q^{\floor{d/p}+(d+1)/2-\floor{d/2p}}$ and so does $H$. The same argument works for $H_{Q_1}$.\end{proof}

Now the problem of estimating $\aod{T_{f,\psi}^r}$ reduces to estimating the number of monic irreducible polynomials of degree $r$ (and $r/2$ if $r$ is even) which can be written in the form $h\equiv g_1(x^p)g_2(x^2)\pmod{x^d}$ for some $g_i\in\fqx,(g_i,x)=1$ (and the same for $x^{d+1}$).
Conjecture \ref{thm1o} follows from heuristics about the number of irreducible polynomials of
given degree falling in the subgroups $H^\bot,H_{Q_1}^\bot$ modulo $x^d,x^{d+1}$ respectively.

The following conjecture, if proven, would settle the case $r<d/4$:

\begin{conj}\label{niceconj} Let $q$ be a power of a prime $p>2$, $d$ a natural number. Let $h\in\fqx$ be prime to $x$, $\deg h=r$ and $r<d/4$. Assume that there exist $g_1,g_2\in\fqx$ s.t. $h\equiv g_1(x^p)g_2(x^2)\pmod{x^d}$. Then for all sufficiently large $d$ there in fact exist $g_3,g_4\in\fqx$ s.t.
$h=g_3(x^p)g_4(x^2)$. In particular if $h$ is irreducible then $h=g_4(x^2)$.\end{conj}
We will give some evidence for Conjecture \ref{niceconj}, namely we will show that in holds for $r< C p\log_q d$ with any constant $0<C<1$.
First we show how Conjecture \ref{niceconj} implies Conjecture \ref{thm1o} for $r<d/4$. First assume that $r$ is odd.
Then by Conjecture \ref{niceconj} and Lemma \ref{hort} we have $\eta(H,r)=\eta(H_{Q_1},r)=0$ and so by Proposition \ref{oddprop} we
have $\aod{T_{f,\psi}^r}=O\lb rq^{-r/6}\rb$. Now assume that $r$ is even. Consider first $\eta(H,r)$. By Conjecture \ref{niceconj} 
and Lemma \ref{hort} any irreducible
polynomial of degree $r$ the residue of which modulo $Q$ is orthogonal to $H$ is of the form $g(x^2)$ with $g\in\fqx$ and $\deg g=r/2$.
The polynomial $g(x^2)$ is irreducible iff $g$ is irreducible and any root of $g$ in $\F_{q^{r/2}}$ is not a square in $\F_{q^{r/2}}$.
The number of such (monic) $g$ is easily seen to be $\frac{q^{r/2}}{2r}+O\lb q^{r/4}\rb$, so $\eta(H,r)=\frac{q^{r/2}}{2r}+O\lb q^{r/4}\rb$.
We also see from Conjecture \ref{niceconj} that $\eta(H_{Q_1},r)=\eta(H,r)$ and $\eta(H,r/2)=\eta(H_{Q_1},r/2)=O\lb q^{r/4}\rb$,
so by Proposition \ref{oddprop} we obtain $\aod{T_{f,\psi}^r}=-e_{2,r}+O\lb rq^{-r/6}\rb$.

We see that to establish Theorem \ref{oddthm} it suffices to prove Conjecture \ref{niceconj} in the case $r<C p\log_q d$.

\subsection{Proof of Conjecture \ref{niceconj} for $r<C p\log_q d$}

Let $q$ be a power of a prime $p>2$ and $r<d$ natural numbers. Let $h\in\fqx$ be a polynomial prime to $x$ with $\deg h=r$. Suppose
that $h$ can be written in the form $h\equiv g_1(x^p)g_2(x^2)\pmod {x^d}$. Since the polynomials of the form $g(x^p)$ modulo $x^d$ (prime to $x$)
form a group we may also write $g_1(x^p)h\equiv g_2(x^2)\pmod{x^d}$ (for a different choice of $g_1$).
Write $g_1=\sum_{i=0}^\ity a_ix^i$ (for sufficiently large $i$ we have $a_i=0$). We may assume $p\le r$, otherwise it is easy to see
that for $g_1(x^p)h$ to be of the form $g_2(x^2)$ modulo $x^d$ the polynomial $h$ itself must be even.
Now take any $C<1$ and assume that $r<Cp\log_q d$. For sufficiently
large $d$ we have $\floor{d/2p}>q^{r/p+1}+r/p+1$.
By the pigeonhole principle for some $j<k<d/p$ we have $a_{j+i}=a_{k+i}$ for $0\le i\le r/p$, with $j,k$ having the same parity.
Now denote by $g_3$ the infinite power series $$g_3(x)=\sum_{i=0}^{k-1}a_ix^i+\sum_{l=0}^\ity\sum_{i=0}^{k-j-1}a_ix^{l(k-j)+j+i}.$$
The first $k+\floor{2r/p}$ coefficients of $g_3$ coincide with those of $g_1$ and then the sequence of coefficients continues periodically with period $k-j$. 
The coefficients of $g_3(x^p)h$ coincide with those of $g_1(x^p)h$ up to the $pk$-th coefficient, after which they continue periodically with
period $(k-j)p$. This is because each coefficient depends on at most $r/p+1$ consecutive coefficients of $g_3$ (and the coefficients of $h$). 
When we multiply $h$ by $g_1(x^p)$ the first $d$ coefficients are zero for odd powers and the same holds for the first $pk$ coefficients
of $g_3(x^p)h$, after which it continues to hold by periodicity (since the period $(k-j)p$ is even). Thus the power series $g_3(x^p)h$ can be written
in the form $g_4(x^2)$ for some power series $g_4(x)$. But $g_3$ is periodic and so must be $g_4$. We then have two rational functions 
$h_1,h_2\in\F_q(x)$ the $x$-adic expansions of which are $g_3,g_4$ respectively and we must have $h=h_1(x)^p/h_2(x^2)$.
Now by the unique factorisation property in $\fqx$ we see that $h$ can be written in this form $h=h_3(x)^ph_4(x^2)$ with $h_i\in\fq[x]$.

\section{The distribution of the number of points on curves in the A-S family}\label{numpoints}

In this section we consider the distribution of the number of points on the curve $C_f$ as $f$ varies uniformly through the family
$\gd$ of all degree $d$ monic polynomials in $\fq[x]$ and
$d\to\infty$ and prove Theorems \ref{t1},\ref{t2},\ref{t3},\ref{t4}.
Throughout this section $r$ is a fixed natural number.

\subsection{Preliminaries}

The number of $\fqr$-rational points on $C(f)$ equals the number of solutions to $F(x,y)=0$ over $\F_{q^r}$ plus one. From the Hilbert 90 theorem or elementary
linear algebra it follows that for $a\in\fqr$ the equation $y^p-y=a$ is solvable in $\fqr$ iff $\tr a=0$, in which case it has exactly $p$ solutions, 
here by $\tr$ we denote the trace map from $\fqr$
to $\F_p$. Thus for a given $x\in\fqr$ the equation $F(x,y)=0$ is solvable iff $\tr f(x)=0$ and in this case it has exactly $p$ solutions
(see section \ref{geom}). We denote by 
$N_r(f)$ the number of solutions to $\tr f(x)=0$ in $\fqr$. It is enough to study the distribution of $N_r(f)$ (the number of points on
the curve is then $pN_r(f)+1$).

From now on we fix $r$. Let $h\in\F_q[x]$ be an irreducible polynomial or degree $e|r$. Its splitting field is the subfield $\F_{q^e}\ss\fqr$ which is isomorphic to $\F_q[x]/h$.
If $a\in\fqr$ is a root of $h$ and $f\in\F_q[x]$ then we denote $\tr_hf=\tr f(a)$ (it does not depend on the choice of the root $a$).
The value of $\tr_hf$ only depends on the residue $f\bmod h$. Denote $$\xi_h(f)=\choice{1,}{\tr_hf=0}{0,}{\tr_hf\neq 0.}$$
Thus we have
\beq\label{2}N_r(f)=\sum_{e|d}e\sum_{\deg h=e}\xi_h(f),\eeq where the inner sum is over monic irreducible $h$ (henceforth $h$ will always denote an irreducible
polynomial in $\F_q[x]$ and summation over $h$ will be always understood in this sense).

If $p$ divides $r/e$ then $\tr_hf=0$ for all $f\in\F_q[x]$. Otherwise exactly $1/p$ of the residues modulo $h$ satisfy $\tr_hf=0$, because
$\tr:\fq[x]/h\to\fp$ is a nonzero $\fp$-linear map.

The following lemmata will be used for the proof of the theorems.

\begin{lem}\label{l1}Let $h_1,...,h_k\in\F_q[x]$ be distinct monic irreducible polynomials, $u=\sum_{i=1}^k\deg h_i$. 
Suppose $d\ge u$. Let $f\in\gd$ be chosen uniformly at random.
Then the values $f\bmod h_1,...,f\bmod h_k$ are distributed uniformly in $\oplus_{i=1}^k\F_q[x]/h_i$.\end{lem}

\begin{proof} The values $f\bmod h_i$ depend only on the residue of $f$ modulo $\prod_{i=1}^kh_i$. If $d\ge u$ then $\gd$ can
be divided into $q^{d-u}$ complete systems of residues modulo any polynomial of degree $d$.\end{proof}

\begin{lem}\label{l2}Suppose that for each natural $m$ we have two sequences of random variables $X_1^{(m)},...,X_{l(m)}^{(m)}$
and $Y_1^{(m)},...,Y_{l(m)}^{(m)}$ satisfying the following conditions:
\ben
\item $X_i^{(m)}$ and $Y_i^{(m)}$ have the same distribution for each $m,i$.
\item $Y_1^{(m)},...,Y_{l(m)}^{(m)}$ are independent for each $m$. 
\item For any fixed $k$ there exists $C_k$ s.t. for $m>C_k$ the variables $X_{i_1}^{(m)},...,X_{i_k}^{(m)}$ are
independent for any $i_1,...,i_k$. \een

Denote $S_m=\sum_{i=1}^{l(m)}X_i\m, T_m=\sum_{i=1}^{l(m)}Y_i\m$. Assume that there are sequences of real numbers $A_m,B_m$ s.t.
the distribution of $A_mT_m+B_m$ weakly converges to a distribution $D$ as $m\to\infty$. Assume further that $D$ is uniquely
determined by its moments. Then $A_mS_m+B_m$ converges in distribution to $D$.\end{lem}

\begin{proof} It is enough to show that for each $k$ the $k$-th moment of $A_mS_m+B_m$ equals the $k$-th moment of $A_mT_m+B_m$ for $m>C_k$.
This follows immediately from the assumed properties, the definition of the $k$-th moment and the multiplicativity of expectation on independent
variables.\end{proof}

\subsection{Proof of the results}

For monic irreducible $h\in\F_q[x]$ and $f\in\F_q[x]$ denote $$N(f)=\sum_{e|r}e\sum_{\deg h=e}\xi_h(f).$$
Lemma \ref{l1} shows that for any distinct $h_1,...,h_k$ the variables $\xi_{h_i}(f),...,\xi_{h_k}(f)$ ($f$ chosen uniformly from $\gd$) are independent for $d\ge\sum\deg h_k$.
We have $\xi_h(f)\sim B(1/p)$ if $(r/\deg h,p)=1$ and $\xi_h(f)\equiv 1$ otherwise (recall that by $B(t)$ we denote the Bernoulli random variable
taking the value 1 with probability $t$ and 0 with probability $1-t$).
If $p|r$ then $\xi_h\equiv 1$ iff $\deg h|(r/p)$ and we have $$\sum_{\deg h|(r/p)}(\deg h)\xi_h(f)=q^{r/p}$$ for any $f$.
Now taking all the monic irreducible $h$ s.t. $\deg h|r$ and noting that the sum of their degrees is $\sum_{e|r}e\nu(q,e)=q^r$ we obtain Theorem 1. 

Now denote by $h_{e,1},...,h_{e,\nu(q,e)}\in\F_q[x]$ the sequence of all monic irreducible polynomials of degree $e$. 
Given a sequence $q(m)=p(m)^{n(m)}$ denote by $$Y_{e,i}, 1\le i\le l(m)=\nu(q(m),e)$$ a set of independent random variables with
$Y_{e,i}\m\sim eB(1/p)$ if $(r/e,p)=1$ and $Y_{e,i}\m\equiv e$ if $p|r$. Taking $X_{e,i}\m=e\xi_{h_{e,i}}(f)$,
we see from Lemma \ref{l1} that the conditions of Lemma \ref{l2} are satisfied for $X_{e,i}\m,Y_{e,i}\m$. 
Thus to establish theorems \ref{t2},\ref{t3},\ref{t4} we only need to show that
$$\sum_{e|r}\sum_{i=1}^{\nu(q(m),e)} Y_{e,i}\m$$ converges in distribution to the limit stated in the theorems. For the rest of this section we omit $m$ from the
notation (it is implicit).

In the setting of Theorem \ref{t2} we get (for $p>r$) a sum of $p$ independent random variables distributed as $B(1/p)$, which converges to the Poissonian distribution
with mean 1. In the setting of Theorem \ref{t3} it is enough to consider the variables corresponding to the irreducible polynomials of degree $r$ and $r/2$ (if the latter is integral), because the number of polynomials of degree $e|r$ and $e<r/2$ is $O\lb q^{r/3}\rb$.
If $r$ is odd then there are $\nu(q,r)=q^{r}/r+O\lb q^{r/3}\rb$ variables $Y_{r,i}$ distributed like $rB(1/p)$, with mean $r/p$ and variance $r^2p(1-1/p)$.
The conclusion now follows from the central limit theorem. 

If $r$ is even and $p>2$ then there are $$\nu(q,r)=\lb q^r-q^{r/2}\rb/r+O\lb q^{r/3}\rb$$ variables $Y_{r,i}$ 
with mean $r/p$ and variance $r^2p^{-1}(1-1/p)$ and $\nu(q,r/2)=q^{r/2}/r+O\lb q^{1/4}\rb$ variables $Y_{r/2,i}$ with mean $r/2p$ and variance $r^2p^{-1}(1-1/p)/4$, which together gives the
same result as for odd $r$. For $p=2$ and $r$ even there are $\lb q^r-q^{r/2}\rb/r+O\lb q^{r/3}\rb$ variables $Y_{r,i}$ distributed as $rB(1/p)$
and $q^{r/2}/r+O\lb q^{1/4}\rb$ variables $Y_{r/2,i}\equiv 1$ and again the conclusion follows from the central limit theorem.

In the setting of Theorem \ref{t4} we can see as above that for odd $r$ the limit distribution of our sum is the same as for $\sum_{i=1}^{\nu(q,r)}Y_{r,i}$.
This is a sum of $\nu(q,r)=q^{r}/r+O\lb q^{r/3}\rb$ independent random variables with distribution $rB(1/p)$. We can group each $p$ such variables
into a single variable with some distribution $S(p)$. The distribution $S(p)$ has mean $r$, variance $r(1-1/p)$ and a bound on the third moment independent
of $p$ (since the variable itself is bounded). 
We have a sum of $q^{r}/(rp)\to\infty$ variables with distribution $S(p)$, so invoking the effective version of the central limit theorem (see 
\cite[\S XVI.5]{feller}) we obtain
the conclusion of Theorem \ref{t4}.

{\bf Acknowledgments.} The author would like to thank Ze\'{e}v Rudnick for suggesting the problems studied in the present work and many helpful
discussions and suggestions in the course of research and writing the present paper. The present work is part of the author's M. Sc. thesis written under
the supervision of Ze\'{e}v Rudnick at Tel-Aviv University.

The author would also like to thank the authors of \cite{uniform} for pointing out a small error in an earlier version of this paper.
Finally the author would like to thank the anonymous referee of this paper for their thorough review and many helpful suggestions for improving
the exposition.

\end{document}